\documentclass{article}
\usepackage[utf8]{inputenc}
\usepackage{amssymb}
\usepackage{amsmath}
\usepackage{algorithm}
\usepackage{algpseudocode}
\usepackage{wrapfig}
\usepackage{longtable}
\usepackage[english]{babel}
\usepackage{booktabs}
\usepackage{xcolor}
\usepackage{multirow}

\usepackage{todonotes}

\usepackage{geometry}
\newcommand{\R}{\mathbb R}

\usepackage{bm}

\usepackage[margins]{trackchanges} 
\addeditor{LP}
\addeditor{AC}
\usepackage{todonotes}
\usepackage{url}

\newtheorem{proposition}{Proposition}[section]
\newtheorem{lemma}[proposition]{Lemma}
\newtheorem{theorem}[proposition]{Theorem}

\newenvironment{proof}[1][Proof]{\begin{trivlist}
\item[\hskip \labelsep {\bfseries #1}]}{\end{trivlist}}

\title{Benders decomposition for congested partial set covering location with uncertain demand}

\author{Alice Calamita$^1$ \and Ivana Ljubi\'c$^2$ \and Laura Palagi$^1$}
\date{$^1$Department of Computer, Control and Management Engineering, Sapienza University of Rome, Via Ariosto 25, 00185 Rome, Italy\\\textit{\{alice.calamita, laura.palagi\}@uniroma1.it}\\
$^2$Department of Information Systems, Decision Sciences and Statistics, ESSEC Business School of Paris, 3 Avenue Bernard Hirsch, 95000 Cergy, France\\\textit{ivana.ljubic@essec.edu}}

\begin{document}

\maketitle

\begin{abstract}
In this paper, we introduce a mixed integer quadratic formulation for the congested variant of the partial set covering location problem, which involves determining a subset of facility locations to open and efficiently allocating customers to these facilities to minimize the combined costs of facility opening and congestion while ensuring target coverage. 
To enhance the resilience of the solution against demand fluctuations, we address the case under uncertain customer demand using $\Gamma$-robustness. We formulate the deterministic problem and its robust counterpart as mixed-integer quadratic problems. We investigate the effect of the protection level in adapted instances from the literature to provide critical insights into how sensitive the planning is to the protection level. Moreover, since the size of the robust counterpart grows with the number of customers, which could be significant in real-world contexts, we propose the use of Benders decomposition to effectively reduce the number of variables by projecting out of the master problem all the variables dependent on the number of customers. 
We illustrate how to incorporate our Benders approach within a mixed-integer second-order cone programming (MISOCP) solver, addressing explicitly all the ingredients that are instrumental for its success. We discuss single-tree and multi-tree approaches and introduce a perturbation
technique to deal with the degeneracy of the Benders subproblem efficiently. Our tailored Benders approaches outperform the perspective reformulation solved using the state-of-the-art MISOCP solver Gurobi on adapted instances from the literature. 
\end{abstract}

\section{Introduction}
The partial set covering location problem (PSCLP) belongs to the class of facility location problems that incorporate a notion of coverage. Coverage is defined by a proximity measure, often the distance or the travel time, establishing whether a potential facility location can serve or cover a certain demand point. The PSCLP aims to minimize the cost of opening the facilities while ensuring the coverage of a specified (partial) amount of the total demand. It derives from the renowned set covering location problem (SCLP), which aims to locate a minimum-cost set of facilities such that all demand points are covered at least once. The PSCLP emerged as a response to the often expensive or impractical solutions produced by the SCLP, allowing for partial coverage.
Also PSCLP has some drawbacks: it does not take into account the congestion at the facilities, which may arise from a sudden increase in the customers' demand. Studying congestion is crucial as it directly impacts the performance (i.e., the quality of service) and efficiency of this type of networks, and sometimes, depending on the application, it could lead to additional cost due to diseconomies of scale (due to, e.g., the employment of additional overtime workers, the use of more expensive materials, or by neglecting or postponing equipment maintenance schedules \cite{harkness2003facility}). 

We, therefore, introduce a novel problem that we denote as the \textit{congested partial set covering location problem} (CPSCLP), consisting of choosing where to locate the facilities among a set of candidate sites, that can satisfy a target (partial) demand in such a way as to minimize the cost of facility opening and congestion. By seeking a more balanced solution, CPSCLP can prevent facilities from being overloaded, thereby minimizing diseconomies of scale and ensuring better resource allocation. Congestion in facility location problems is typically modeled through a convex quadratic term in the objective function (see, e.g., \cite{desrochers1995congested,fischetti2016benders}). This term acts similarly to a penalty function associated with capacity constraints, as it penalizes every additional unit being served by a given facility. Hence, taking into account the minimization of congestion can be seen as imposing a limited capacity on the facilities, as claimed in \cite{desrochers1995congested}.

Another key aspect to consider is that in most real-world applications, accurately estimating demand can be challenging due to its inherent variability or lack of historical data. If the estimate is not correct and the demand deviates from the expected value, it could lead to an optimal solution that is actually impracticable or infeasible. To mitigate the sensitivity of the solution with respect to changes in the problem parameters, we address the CPSCLP under the assumption of uncertain customer demand. Specifically, we propose to deal with data uncertainty using the approach known as $\Gamma$-robustness, as introduced in \cite{bertsimas2003robust}. This approach aims to create a reliable and efficient network architecture that is robust against demand changes. In particular, based on the assumption that nature is restricted in its behaviour, we protect against the case in which deviations from expected demand will occur for at most $\Gamma$ customers.

\paragraph{Literature overview}
The problem we address emerges at the intersection of two optimization problems: the PSCLP and the congested facility location problem (CFLP). Despite their practical relevance, these problems have received little attention in the scientific literature. 

The minimum cost covering problem traces back to Hakimi's work \cite{hakimi1965optimum} in 1965, where the problem of locating the minimum number of police officers so that everyone is within a given distance from an officer is introduced; Hakimi suggests a solution procedure based on enumeration.
The first integer programming formulation of the problem was proposed in \cite{toregas1971location} to solve the problem of locating emergency service facilities in a discrete space. Then, in \cite{walker1974using}, a heuristic is suggested to assign ladder trucks to fire stations and the problem is formulated as a minimum cost covering problem. 

PSCLP, specifically, was introduced by Daskin and Owen \cite{daskin1999two} in 1999. The authors propose an approach based on a Lagrangian heuristic. In 2019, an exact approach based on Benders decomposition was provided in \cite{cordeau2019benders}, and several large-scale instances were made available by the authors as a benchmark for future works. More recently, in \cite{chen2023efficient}, five customized presolving methods have been discussed to enhance the capability of employing mixed-integer programming (MIP) solvers in solving PSCLPs.

Other studies have considered partial set covering problems in different contexts than location, such as a mining application in \cite{bilal2014iterated}, and a related problem in which customers whose distance falls between a lower and an upper bound from their nearest facility are only partially covered \cite{berman2003gradual}. 

An alternative version to PSCLP is the maximal covering location problem (MCLP), that aims for the subset of facilities
maximizing the coverage while respecting a budget constraint. The MCLP was formulated for the first time by Church and ReVelle \cite{church1974maximal} in 1974 as a 0-1 linear programming and in \cite{megiddo1983maximum} was proved to be NP-hard. There are two versions of this problem in the literature, one imposing an upper bound on the number of open facilities -- whose LP relaxation could be integral for relatively small size instances (see \cite{snyder2011covering})-- and the other using a budget constraint -- whose LP relaxation usually leads to more fractional solutions. The MCLP received more attention than PSCLP. However, despite the good quality of the LP relaxation, the exact approaches proposed are few. 
Indeed, mostly heuristic and metaheuristic algorithms have been proposed for this problem. Among them, we mention: greedy procedures in \cite{church1974maximal, resende1998computing}; heuristics based on the Lagrangian relaxation of the constraints in \cite{downs1996exact,galvao2000comparison,galvao1996lagrangean,lorena2002lagrangean}; a decomposition heuristic in \cite{senne2010decomposition};   metaheuristics in \cite{revelle2008solving,atta2018solving,maximo2017intelligent,maximo2019intensification,maximo2023hybrid,atta2023improved}.
As for the exact approaches, in \cite{downs1996exact}, a branch-and-bound framework using Lagrangian relaxations to dualize the covering constraints is investigated. 
Finally, a Benders decomposition method has also been provided   for this problem in \cite{cordeau2019benders}. 

In this paper, we model congestion  through convex quadratic functions, as it was done for the congested facility location problem (CFLP). We therefore review this class of problems. The CFLP was introduced in 1995 by Desrochers et al. \cite{desrochers1995congested}, inspired by \cite{krarup1987optimal}, in which a similar formulation was proposed to illustrate a brief example involving skiers waiting for chairlifts. The work of \cite{desrochers1995congested} provides a column generation embedded in a branch-and-bound scheme, and reports optimal solutions for very small instances of the problem. 
In \cite{harkness2003facility}, different MIP formulations for the case of convex and piece-wise linear production cost functions are compared. Other contributions include the two master theses \cite{lu2010facility} and \cite{cselfun2011outer} investigating the problem, and the article \cite{fischetti2016benders}, where the authors propose a Benders decomposition method that is effective even though the subproblem is not separable. A branch and bound algorithm based on Lagrangian relaxation and subgradient optimization is suggested in \cite{christensen2021fast}. For a recent study on how off-the-shelf MIP solvers behave on CFLP instances modeled throw mixed-integer second-order cone programs, see \cite{avella2023off}. 

There are several other studies modeling facility congestion using queuing theory (some of them also considering
 uncertainty in the data incorporated in the models) that are not directly linked to our models. We therefore refer the reader to the survey \cite{boffey2007review} and the literature overviews provided in \cite{aboolian2012profit,zhang2009incorporating} for further information.

For what concerns data uncertainty, there are several contributions in stochastic or robust facility location (e.g., see the review in \cite{snyder2006facility}). However, none of them specifically addresses the PSCLP.

\paragraph{Main motivation and our contribution}
A careful look into the literature suggests a lack of contributions in problems of congested facility location and partial set covering location. Recent works \cite{cordeau2019benders,fischetti2016benders,fischetti2017redesigning} showed promising results given by exact methods based on Benders decomposition for the deterministic PSCLP and CFLP (both the capacitated and uncapacitated cases). However, no consideration is given to the inherent volatility of the parameters used to model these specific problems. Motivated by the theoretical and practical relevance, we address the robust and congested variant of PSCLP, which considers the minimization of the congestion at the facilities and the protection against the changes in demand, and investigate the use of Benders decomposition to solve large instances of this problem.

The contributions of this paper can be summarized as follows:
\begin{enumerate}
    \item we state and formulate the congested partial set covering location problem under the deterministic setting;
    \item we consider the uncertainty in the customer demand and formulate the robust counterpart of the problem leveraging the $\Gamma$-robustness theory. The demand changes affect both the congestion function and the covering constraint. We show how the nonlinear protection functions can be linearized;
    \item we consider the case of quadratic convex congestion cost functions and apply the perspective reformulation to the robust counterpart of the problem, aiming for a tighter formulation. We then develop a tailored solution approach for the resulting model based on Benders decomposition;    
    \item we illustrate how to integrate our Benders approach within a MISOCP solver, explicitly addressing all the ingredients that are instrumental for successful implementation. Although Benders decomposition is a traditional technique, our approach is innovative as it leverages the implementation of callback functions in combination with quadratic constraints, a novel option provided by modern MISOCP solvers. We assess both single-tree and multi-tree Benders approaches through a comprehensive set of (adapted) instances from existing literature, aiming to validate their efficiency and compare their performance with a state-of-the-art MISOCP solver;
    \item we test the effect of the protection level in the coverage and load (together and separately) to provide critical insights into how sensitive the planning is to the protection level, offering a more comprehensive understanding of the problem.
\end{enumerate}

The paper is structured as follows. In Section \ref{ch_formulation}, we introduce the deterministic formulation of the problem and its robust counterpart, explaining how to account for uncertainties in customer demand using the $\Gamma$-robustness approach. 
Section \ref{ch_benders} focuses on the Benders decomposition method: in this section, we provide the master and subproblem formulations, and derive the expression of the Benders cut. 
In Section \ref{ch_epsilon_technique}  we show a cut-strengthening technique to deal with the degeneracy of the Benders subproblem. 
Section \ref{ch_implementation} provides insights into the integration of Benders approach in a MISOCP solver. We describe in detail both the single-tree and multi-tree approaches. 
In Section \ref{ch_results}, we report the sensitivity analysis testing the effect of $\Gamma$ on the optimal solutions and the computational experiments comparing our tailored Benders to a state-of-the-art solver. Conclusions are given in Section \ref{ch_conclusions}.

\section{Problem Setting} \label{ch_formulation}
We are given a set $I$ of potential facility locations with opening cost $f_i \geq 0$ for $i \in I$, and a set $J$ of customer locations such that each customer location $j \in J$ is associated with a demand $d_j \geq 0$. From now on, we refer for short to facility and customer, omitting the word location.
For each customer $j$, we are also given a subset $I(j) \subseteq I$ of facilities that can cover $j$, i.e., that can fully serve the demand $d_j$. 
Similarly, for $i \in I$,
let $J(i)$  be the subset of customers that can be covered by $i$. More generally, for a subset of facilities $N \subseteq I$, let 
 $J(N) \subseteq J$ be the subset of customers that can be covered by an $i\in N$.

Given a parameter $0 < D \leq \sum_{j \in J} d_j$, we aim to identify a subset of facilities to open in order to ensure that the total served customer demand is at least $D$, while minimizing the overall costs, given by facility opening expenses and congestion costs.
To model congestion cost, we follow what is done in the literature of CFLP \cite{desrochers1995congested,fischetti2016benders,krarup1987optimal}, in which a nonlinear cost function $F_i$ is used to penalize each additional unit of demand served by a given facility $i$.
For each facility $i\in I$, let $F_{i}:\mathbb R \to \mathbb R_+ $ be a convex and non-decreasing function representing the congestion cost for each unit of served demand.

For each $i \in I$, the binary variable $y_i$ is set to one if facility $i$ is open and to zero otherwise.
For each $i \in I$ and $j \in J$, the continuous allocation variable $0\le x_{ij} \le 1$ denotes the fraction of the demand of customer $j \in J$ served by facility $i \in I$. 

\subsection{The deterministic problem formulation}
The most general deterministic CPSCLP can be formulated as a mixed-integer nonlinear program (MINLP) as follows 
\begin{align} \displaystyle
 {\min} \quad  & \sum_{i \in I} f_{i}y_{i} +\sum_{i \in I} F_{{i}}\left(  \sum_{j \in J}d_j x_{ij} \right)\label{objective_partial_set_cov}\\
 {s.t.} &\quad \sum_{i \in I} \sum_{j \in J} d_j x_{ij} \geq D \label{demand_constraint}\\
&\sum_{i \in I(j)} x_{ij} \leq 1 && j \in J \label{assignment_constraint}\\
& 0 \leq x_{ij} \leq y_i && j \in J, i \in I(j) \label{linking_constraint}\\
&y_{i} \in \{0,1\} && i \in I. \nonumber
\end{align}

The objective \eqref{objective_partial_set_cov} aims to minimize the sum of facility opening and congestion costs. For the sake of simplicity, we assign equal weights to the two conflicting objectives. However, different (non-negative) weights can be easily incorporated without altering the structure of the problem.
Constraint \eqref{demand_constraint} ensures the total customer demand served is at least $D$. Assignment constraints \eqref{assignment_constraint} make sure that the fraction of covered demand of a customer does not exceed the unit. Finally, constraints \eqref{linking_constraint} ensure that allocation to a facility is only possible if it is open.

This problem is NP-hard since it is a generalization of the traditional set covering location problem, which is NP-hard \cite{daskin1999two}.

\subsection{The robust counterpart of the problem}

 In real-world scenarios, customer demand often varies or is difficult to estimate. To capture this uncertainty, we assume that each entry $d_j$, $j \in J$ of the vector of demand $d$ is modeled as an independent, symmetric and bounded random variable (with unknown distribution) $\tilde d_j$, $j \in J$ that takes values in $[d_j - \hat d_j, d_j + \hat d_j]$. We allow the possibility that the deviations from the nominal values
  could also be zeros, i.e. that $\hat d_j=0$ for some $j \in J$.
We adopt the notion of protection introduced by Bertsimas and Sim in \cite{bertsimas2003robust} known as $\Gamma$-robustness, which assumes that only a subset of the components of $d$ will deviate from their nominal values, adversely affecting the solution. Hence, we introduce an integer number $\Gamma$, taking values in the interval $[0, |J|]$, that limits the number of demand deviations. Parameter $\Gamma$ controls the level of robustness against the solution: if $\Gamma=0$, we completely ignore the uncertainty (deterministic setting), while if $\Gamma=|J|$, we are considering all possible demand deviations (which is the most conservative strategy).

We note that the vector of demand $d$ is involved in the modeling of congestion and coverage affecting both optimality and feasibility. By protecting against the uncertainty, we mean that we are interested in finding 
an optimal solution that:
\begin{enumerate}
    \item optimizes against all scenarios under which up to $\Gamma$ demand coefficients can vary in such a way as to maximally influence the objective; the worst-case scenario is given by the largest increase of demand;
    \item is protected against all cases in which up to $\Gamma$ demand coefficients change affecting the feasibility; the worst-case scenario is given by the largest decrease of demand.
\end{enumerate}
We observe that the two worst-case realizations play against each other.

We now introduce a robust counterpart of the problem, that optimizes against the worst-case realizations under demand uncertainty and reads 
\begin{equation}\label{robust_nonconvex}
\begin{aligned}
{\min} \quad  & \sum_{i \in I} f_{i}y_{i} 
+\sum_{i \in I} F_{i}\left(  \sum_{j \in J}d_j x_{ij} +\max_{\{S: \ S \subseteq  J, |S| \leq \Gamma\}}\left\{
 \sum_{j \in S} \hat d_j x_{ij} \right\}\right)\\
{s.t.} \quad & \sum_{i \in I} \sum_{j \in J} d_j x_{ij} - \max_{\{S: \ S \subseteq  J, |S| \leq \Gamma\}}\left\{\sum_{i \in I} \sum_{j \in S} \hat d_j x_{ij} \right\} \geq D \\
&\sum_{i \in I(j)} x_{ij} \leq 1 \hspace{7.02cm} j \in J \\
& 0 \leq x_{ij} \leq y_i \hspace{7.15cm} j \in J, i \in I(j)\\
& y_{i} \in \{0,1\} \hspace{7.4cm} i \in I
\end{aligned}
\end{equation}
where, for a given $\Gamma$ and allocation choice $x_{ij}$, the objective function is taking into account the sum of the $\Gamma$ largest deviations in case the demand is increasing from the nominal value, whereas in the covering constraint, we are considering the sum of the $\Gamma$ largest deviations in case the demand is decreasing from the nominal value.

We now introduce the auxiliary aggregated variables $v_i\geq 0$ for $i \in I$, 
denoting the total demand served by facility $i$ (also known as facility load). Then,
Problem \eqref{robust_nonconvex} becomes
\begin{subequations}\label{new_def_v}
\begin{align}
 \displaystyle{\min_{x \ge 0, y \in \{0,1\}^{|I|}}} \quad & \displaystyle\sum_{i \in I} f_{i}y_{i}   + \sum_{i \in I} F_{ i}(v_i) \\
{s.t.} \quad & \displaystyle v_i \geq \sum_{j \in J} d_j x_{ij} + \max_{\{S: \ S \subseteq  J, |S| \leq \Gamma\}}\left\{\sum_{j \in S} \hat d_j x_{ij} \right\} && i \in I \label{constraint_on_v}
\\
&  \displaystyle\sum_{i \in I} \sum_{j \in J} d_j x_{ij} - \max_{\{S: \ S \subseteq  J, |S| \leq \Gamma\}}\left\{\sum_{i \in I} \sum_{j \in S} \hat d_j x_{ij} \right\} \geq D \label{constraint_on_D}\\
& \displaystyle\sum_{i \in I} x_{ij} \leq 1 && j \in J \\
& \displaystyle  x_{ij} \leq y_i && j \in J, i \in I(j)
\end{align}
\end{subequations}

Following  \cite{bertsimas2003robust}, by applying strong duality to the inner maximization robust terms, we 
 can derive the following equivalent MINLP formulation of this problem
\begin{subequations}\label{whole_extended}
\begin{align}
\displaystyle
{\min_{(x, \tau, \rho, \pi, \sigma)\geq 0,\, y \in \{0,1\}^{|I|}}} \quad &  \sum_{i \in I} f_{i}y_{i}  + \sum_{i \in I} F_{ i}(v_i)  \nonumber\\
{s.t.} \quad & v_i - \sum_{j \in J} d_j x_{ij} - \left(\Gamma \rho_i + \sum_{j \in J} \sigma_{ij}\right) \geq 0 && i \in I \label{first}\\
& \sum_{i \in I} \sum_{j \in J} d_j x_{ij} - \left(\Gamma \tau + \sum_{j \in J} \pi_j\right) \geq D \label{second}\\
& \tau + \pi_j \geq \sum_{i \in I} \hat d_j x_{ij} && j \in J\\
& \rho_i + \sigma_{ij} \geq  \hat d_j x_{ij} && j \in J, i \in I\\
&\sum_{i \in I} x_{ij} \leq 1 && j \in J \\
& x_{ij} \leq y_i && j \in J, i \in I(j) \label{last}
\end{align}
\end{subequations}

We can state the following theorem, whose proof is in the Appendix \ref{sec:appendix_proof}.
\begin{theorem}
\label{th:appendix}
    Formulation \eqref{whole_extended} is equivalent to \eqref{new_def_v}.
\end{theorem} Problem \eqref{whole_extended} is a much more computationally tractable formulation. 
Indeed, the problem is now formulated as a mixed-integer program with linear constraints and a convex objective function.
\subsubsection{The perspective reformulation}
Following the work of \cite{desrochers1995congested,fischetti2016benders}, we consider the specific case where the congestion function is given by 
the convex function $F_{i}(t)=a_{i}t^2 + b_{i}t$ for $i \in I$, with $a_{i}$ and $b_{i}$ non-negative input coefficients. 
Then, the robust counterpart of the problem, formulated in \eqref{whole_extended}, reads
\begin{align}
\displaystyle
{\min_{(x,\tau, \rho, \pi, \sigma)\geq 0,\, y \in \{0,1\}^{|I|}}} \quad &  \sum_{i \in I} f_{i}y_{i} +  \sum_{i \in I} b_{i} v_i + \sum_{i \in I} a_{i}  v_i^2 \nonumber\\
{s.t.} \quad
& \eqref{first}-\eqref{last} \nonumber.
\end{align}
which is a mixed-integer (convex) quadratic program (MIQP) with binary variables. Therefore, we can
derive the so-called \textit{perspective reformulation}
\cite{frangioni2009computational,gunluk2012perspective}
\begin{subequations}\label{whole_perspective}
\begin{align}
\displaystyle
{\min_{(x,u,v,\tau, \rho, \pi, \sigma)\geq 0,\, y \in \{0,1\}^{|I|}}} \quad &  \sum_{i \in I} f_{i}y_{i} + \sum_{i \in I} b_i v_i + \sum_{i \in I} a_{i} u_i \nonumber\\
{s.t.} \quad & \eqref{first}-\eqref{last} \nonumber\\
& v_i^2 \leq u_i y_i && i \in I \label{perspective_cuts}
\end{align}
\end{subequations}
in which the quadratic term $v_i^2$ in the objective function is replaced by a non-negative variable $u_i$.
Constraints \eqref{perspective_cuts} are rotated second-order cone constraints (hence still convex) as the right-hand side is the product of non-negative variables. They are obtained by the strengthening of the $u_i$ variables through constraints $v_i^2 \leq u_i$, by replacing the convex function $v_i^2$ with its perspective defined by $y_i(v_i/y_i)^2$ if $y_i > 0$, and zero if $y_i = 0$; see \cite{gunluk2012perspective} for details.
It is well-known (see e.g., \cite{akturk2009strong,frangioni2006perspective,frangioni2007sdp,frangioni2008tighter,gunluk2008perspective,gunluk2012perspective}) that the continuous relaxation of a perspective reformulation produces stronger bounds than the bounds given by the continuous relaxation of the original formulation. Consequently, we can state that the original formulation \eqref{whole_extended}, with $F_{i}(t)=a_{i}t^2 + b_{i}t$, lies in an extended space compared to the formulation \eqref{whole_perspective}, and for this reason, we denote it as the \textit{extended formulation}.

\section{Benders Decomposition}\label{ch_benders}

Both the extended and the perspective formulations have a number of variables and constraints that depend on the number of customers $|J|$ and facilities $|I|$.
In real world context, the number of customers can be
quite large, affecting the solution time of problems \eqref{whole_extended} or \eqref{whole_perspective}. In this regard, we propose to use Benders decomposition, a classical method for mixed-integer linear programming introduced by Benders \cite{benders1962partitioning} in 1962. Our idea is to reduce as much as possible the number of variables in the problem, by projecting out of the master problem (at least) all the variables that depend on the number of customers. This will lead to a boosting in the solution time, as we will see in Section \ref{ch_results}. 
Specifically, we investigate the projecting out of the continuous variables $(x, \tau, \rho, \pi, \sigma)$, leading to a convex master problem containing the cones coming from the perspective reformulation and the complete objective function, and an LP subproblem. The advantage of having the cones in the master formulation is that the master problem has a very tight relaxation, which is amenable when dealing with Benders decomposition as it can substantially reduce solution times. We also observe that the subproblem -- which is solved several times to generate valid cuts for the master -- is an LP that can be therefore efficiently solved by any state-of-the-art solver.

\subsection{The master problem}

From the formulation \eqref{whole_perspective}, we project out the continuous variables $(x,\rho,\tau,\pi,\sigma)$. The master
problem is then defined on the variables $(y,u,v)$ and is given by
\begin{equation}\label{pb:master}
\begin{array}{rlrr}
\displaystyle
{\min_{(u,v) \geq 0,\, y \in \{0,1\}^{|I|}}}  & \displaystyle \sum_{i \in I} f_{i}y_{i} + \sum_{i \in I} b_i v_i + \sum_{i \in I} a_i u_i\\
{s.t.} \quad & \phi(y,v) \leq 0 \\
& v_i^2 - u_i y_i \leq 0 && i \in I
\end{array}\end{equation}
where $\phi(y,v)$ is a convex function that evaluates the feasibility of the vector $(y,v)$ for the constraints \eqref{first}-\eqref{last}. It can be approximated by linear cuts to be generated on the fly, known as Benders cuts, that are valid for any given vector $(y, v)$. 
Usually, two types of Benders cuts are used: the optimality cuts and the feasibility cuts. In our approach, there is no need for Benders optimality cuts, as all the variables appearing in the objective function $(y,v,u)$ belong to the master problem. Hence, we will only use Benders feasibility cuts to discard infeasible points.
The number of these cuts is exponential, making enumerating them all at once impractical.
Since only some of them are necessary to find an optimal solution, they will be dynamically separated by the decomposition approach. Consequently, the master problem will contain only a subset of Benders cuts, belonging to a so-called \emph{relaxed master problem}.

The relaxed master problem is a mixed-integer second order cone program (MISOCP), thus a convex MINLP. It  can be solved as a mixed-integer linear program by a branch-and-cut approach where the integrality requirement on $y$ is relaxed and linear outer-approximations of the
conic constraints are generated on the fly or a nonlinear programming solver is used for the underlying NLP problem at every node \cite{gurobiWebsite}.

\subsection{The subproblem}

Given a master solution $(\bar y,\bar v,\bar u)$, the subproblem is then given by
\begin{align*}
\displaystyle
\phi(\bar y,\bar v) = {\min_{(x,\rho, \tau,\pi,\sigma)\geq 0}} \quad &  0\\
{s.t.} \quad &  \sum_{j \in {J}} d_j x_{ij} + \Gamma \rho_i + \sum_{j \in {J}} \sigma_{ij} \leq  \bar v_i  && i \in I \\
& \sum_{j \in J} \sum_{i \in {I}} d_j x_{ij} - \Gamma \tau - \sum_{j \in J} \pi_j \geq D\\
& \tau + \pi_j - \sum_{i \in {I}} \hat d_j x_{ij} \geq 0 && j \in J\\
& \sigma_{ij} + \rho_i - \hat d_j x_{ij} \geq 0  && j \in J, i \in I(j)\\
&\sum_{i \in {I}} x_{ij} \leq 1 && j \in J\\
& x_{ij} \leq \bar y_i && i \in I(j), j \in J.
\end{align*}
which is a linear programming problem. The problem above can be feasible and in this case there is no need to generate a Benders cut, or infeasible. In the latter case a Benders feasibility cut is produced, i.e., a cutting plane 
$\phi(\bar y,\bar v) \leq 0$
that discards the master solution $(\bar y,\bar v,\bar u)$ that leads to the infeasibilty of the subproblem. Solving a subproblem that could potentially be infeasible may lead to computational issues. An infeasible primal problem means that in the dual space we are optimizing over an unbounded cone pointed at zero. Among the successful strategies to overcome these difficulties are the normalization techniques that consist  of solving the dual LP over a bounded polyhedron. There is abundant literature on different normalization techniques for Benders feasibility cuts, see, e.g., \cite{cordeau2019benders, fischetti2010note, ljubic2012exact, magnanti1981accelerating, papadakos2008practical}.
We use a very natural approach exploiting the fact that the solution $(\bar y, \bar v, \bar u)$ of the relaxed master problem is infeasible for the subproblem if and only if the demand covered by $\bar y$ is strictly less than $D$. Indeed, in our case, the covering constraint forms the irreducible infeasible subsystem, which is the minimal subset of constraints whose removal makes the problem feasible. Hence, instead of solving a feasibility LP given by the formulation above, we search for the maximum coverage. The resulting subproblem is the following LP
\begin{align*}
\displaystyle
\phi^{\prime}(\bar y,\bar v) = {\max_{(x,\rho, \tau,\pi,\sigma)\geq 0}} \quad &  \sum_{j \in J} \sum_{i \in {I}} d_j x_{ij} - \Gamma \tau - \sum_{j \in J} \pi_j\\
{s.t.} \quad &  \sum_{j \in {J}} d_j x_{ij} + \Gamma \rho_i + \sum_{j \in {J}} \sigma_{ij} \leq  \bar v_i  && i \in I\\
& \tau + \pi_j - \sum_{i \in {I}} \hat d_j x_{ij} \geq 0 && j \in J\\
& \sigma_{ij} + \rho_i - \hat d_j x_{ij} \geq 0  && j \in J, i \in I(j)\\
&\sum_{i \in {I}} x_{ij} \leq 1 && j \in J\\
& x_{ij} \leq \bar y_i && j \in J, i \in I(j),
\end{align*}
in which we observe that for every given $(\bar y, \bar v)$, the vector $(x,\rho,\tau,\pi,\sigma)=(0,0,0,0,0)$ is a feasible solution, i.e., this problem is always feasible and bounded as $(x, \tau, \pi)$ are bounded. If the optimal value of the subproblem $\phi^{\prime}(\bar y,\bar v)$ is greater or equal to $D$, the master solution is feasible and we do not need to generate any Benders cut. If, instead, $\phi^{\prime}(\bar y,\bar v)$ is strictly less than $D$, the master solution is infeasible as it does not guarantee the required coverage and we do generate a Benders feasibility cut. The Benders cut is then given by $\phi^{\prime}(y, v)\geq D$ (and no more by $\phi(y, v) \leq 0$) for every solution $(y, v)$ of the relaxed master problem.
We notice that the extreme points of the polyhedron describing the problem $\phi^{\prime}(\bar y, \bar v)$ correspond to extreme rays of the subproblem associated to $\phi(\bar y, \bar v)$ in which the dual variable of the covering constraint is fixed to one.

Because of concavity, $\phi^{\prime}(\cdot)$ can be overestimated by a supporting hyperplane at $(\bar y, \bar v)$, so the following linear cut is also valid
$$\phi^{\prime}(\bar y, \bar v) + \xi_y (\bar y,\bar v)^T (y - \bar y) + \xi_v (\bar y,\bar v)^T (v - \bar v) \geq \phi^{\prime}(y, v) \geq D$$
where $\xi_y(\bar y,\bar v),\,\xi_v(\bar y, \bar v)$ denote the subgradients of $\phi^{\prime}$ with respect to $y, v$ in $(\bar y, \bar v)$. 

Depending on the problem, the computation of the subgradients could be heavy. Therefore, we introduce a
simple reformulation of the Benders subproblem that makes their calculation straightforward, following what was done, e.g., in \cite{fischetti2016benders}. The reformulation of the subproblem reads
\begin{subequations}\label{final_subproblem}
\begin{align}
\displaystyle
\phi^{\prime}(\bar y,\bar v) = \max_{y,v,(x,\rho, \tau,\pi,\sigma)\geq 0} \quad &  \sum_{j \in J} \sum_{i \in {I}} d_j x_{ij} - \Gamma \tau - \sum_{j \in J} \pi_j \nonumber\\
{s.t.} \quad &  \sum_{j \in {J}} d_j x_{ij} + \Gamma \rho_i + \sum_{j \in {J}} \sigma_{ij} - v_i \leq 0 && i \in I \nonumber\\
& \tau + \pi_j - \sum_{i \in {I}} \hat d_j x_{ij} \geq 0 && j \in J \nonumber\\
& \sigma_{ij} + \rho_i - \hat d_j x_{ij} \geq 0  && j \in J, i \in I(j) \nonumber\\
&\sum_{i \in {I}} x_{ij} \leq 1 && j \in J \nonumber\\
& x_{ij} - y_i \leq 0 && j \in J, i \in I(j)\nonumber \\
& y_i = \bar y_i && i \in I \label{fixing_y}\\
& v_i = \bar v_i && i \in I \label{fixing_v}
\end{align}
\end{subequations}
where we keep the master variables $y$ and $v$ as variables of the subproblem as well, and we apply variable fixing through \eqref{fixing_y}--\eqref{fixing_v}. 
By construction, the subgradients are simply $\xi_y(\bar y,\bar v) = \bar r_y$ and $\xi_v(\bar y,\bar v) = \bar r_v$, where $\bar r_y$ and $\bar r_v$ are the vectors of reduced costs associated with $y$ and $v$. 
The Benders cut reads
\begin{equation}
    \label{bcut}
    \phi^{\prime}(\bar y, \bar v) + \sum_{i \in I} \bar r_{y_i} (y_i - \bar y_i) +  \sum_{i \in I} 
 \bar r_{v_i}(v_i - \bar v_i) \geq D
\end{equation}
where each component of the reduced cost $\bar r_{y_i}$ (or $\bar r_{v_i}$) defines an upper bound on the increase of the objective function $\phi^{\prime}(\bar y, \bar v)$ when $\bar y_i$ (or $\bar v_i$) increases.

\section{Addressing degeneracy}
\label{ch_epsilon_technique}

When applying Benders decomposition, a highly degenerate subproblem admitting several optimal solutions can result in the generation of shallow Benders cuts. This slows down the convergence of the Benders decomposition, requiring the addition of many cuts, which do not improve the bound that much. There are many techniques proposed in the literature to address degeneracy (see, e.g., \cite{rahmaniani2017benders}). 
Inspired by the perturbation technique proposed in \cite{fischetti2017redesigning} to accelerate the convergence of a cut loop (i.e., the cut separation at the root node of the branching tree), we adapt this technique, that we will refer to as $\epsilon$-\textit{technique}, leading to stronger cuts. The result is an accelerated convergence and a fewer number of generated cuts, as shown from the computational experience provided later in Section \ref{ch_results}. 

We start by interpreting the Benders cut from the LP primal-dual perspective, and then show how to apply this perturbation to derive valid Benders cuts for the non-perturbed subproblem. 
Consider the \textit{non-perturbed} subproblem formulated as follows
\begin{equation}\label{subproblem_fixed}
\begin{aligned}
\phi^{\prime}(\bar y,\bar v) = {\max_{y,v,(x,\rho, \tau,\pi,\sigma)\geq 0}} \quad &  \sum_{j \in J} \sum_{i \in {I}} d_j x_{ij} - \Gamma \tau - \sum_{j \in J} \pi_j \\
{s.t.} \quad &  \sum_{j \in {J}} d_j x_{ij} + \Gamma \rho_i + \sum_{j \in {J}} \sigma_{ij} - v_i \leq 0 &\quad& i \in I \\
& \tau + \pi_j - \sum_{i \in {I}} \hat d_j x_{ij} \geq 0 &\quad& j \in J \\
& \sigma_{ij} + \rho_i - \hat d_j x_{ij} \geq 0  &\quad& j \in J, i \in I(j) \\
&\sum_{i \in {I}} x_{ij} \leq 1 &\quad& j \in J && (\alpha_j)\\
& x_{ij} - y_i \leq 0 &\quad& j \in J, i \in I(j) \\
& y_i = \bar y_i &\quad& i \in I && (\beta_i)\\
& v_i = \bar v_i &\quad& i \in I. && (\gamma_i) 
\end{aligned}
\end{equation}
where by $\alpha,\, \beta$ and $ \gamma$ we denote the dual variables (or the Lagrangian multipliers).
Given the optimal dual solution $(\bar \alpha, \bar \beta, \bar \gamma)$, the optimal value of the dual of \eqref{subproblem_fixed} is
\begin{equation*}
    \phi^{\prime}(\bar y,\bar v) =  \sum_{j \in J} \bar \alpha_j + \sum_{i \in I}  \bar y_i \bar \beta_i + \sum_{i \in I} \bar v_i \bar \gamma_i.
\end{equation*}
We can get the expression of the Benders cut from the LP dual by imposing that the dual objective at the optimal solution is at least the target coverage, namely 
\begin{equation*}
\sum_{j \in J} \bar \alpha_j + \sum_{i \in I} \bar \beta_i y_i  + \sum_{i \in I} \bar \gamma_i v_i \geq D.
\end{equation*}
The dual problem is highly degenerate as many objective coefficients, i.e., many components of $\bar y$ and $\bar v$ can be zero, resulting in several free components of the dual variables $\beta$ and $\gamma$ and many equivalent
optimal solutions. To get the strongest Benders cut, we need the dual multipliers appearing in the cut to take the smallest values possible. To get this, we can apply the $\epsilon$-technique that consists of replacing the zero objective coefficients of $\bar y$ and $\bar v$ with a sufficiently small $\epsilon > 0$ and solving the resulting problem. This induces the dual model to minimize also the components of the dual variables $\beta$ and $\gamma$ associated with the originally zero objective coefficients. 

Unfortunately, solving the dual of the subproblem often leads to more numerical issues than solving its primal, hence we explain how to apply this technique to the primal subproblem \eqref{subproblem_fixed}. Specifically, given a sufficiently small $\epsilon > 0$, applying the $\epsilon$-technique to the primal problem consists of replacing $\bar y$ and $\bar v$ with 
\begin{equation*}
    \bar y^{\epsilon}_i = \begin{cases}
    \bar y_i & \text{if } \bar y_i > 0\\
\epsilon & \text{if } \bar y_i = 0\end{cases}
\qquad \text{and} \qquad
    \bar v^{\epsilon}_i = \begin{cases}
    \bar v_i & \text{if } \bar v_i > 0\\
\epsilon & \text{if } \bar v_i = 0\end{cases}
\end{equation*}
for each $i \in I$. Then, the \textit{perturbed} subproblem reads
\begin{equation}\label{subproblem_perturbed}
\begin{aligned}
\phi^{\prime}(\bar y^{\epsilon},\bar v^{\epsilon}) = {\max_{y,v,(x,\rho, \tau,\pi,\sigma)\geq 0}} \quad &  \sum_{j \in J} \sum_{i \in {I}} d_j x_{ij} - \Gamma \tau - \sum_{j \in J} \pi_j \\
{s.t.} \quad &  \sum_{j \in {J}} d_j x_{ij} + \Gamma \rho_i + \sum_{j \in {J}} \sigma_{ij} - v_i \leq 0 &\quad& i \in I \\
& \tau + \pi_j - \sum_{i \in {I}} \hat d_j x_{ij} \geq 0 &\quad& j \in J \\
& \sigma_{ij} + \rho_i - \hat d_j x_{ij} \geq 0  &\quad& j \in J, i \in I(j) \\
&\sum_{i \in {I}} x_{ij} \leq 1 &\quad& j \in J  && (\alpha_j^{\epsilon})\\
& x_{ij} - y_i \leq 0 &\quad& j \in J, i \in I(j) \\
& y_i = \bar y^{\epsilon}_i &\quad& i \in I && (\beta_i^{\epsilon})\\
& v_i = \bar v^{\epsilon}_i &\quad& i \in I.  && (\gamma_i^{\epsilon})
\end{aligned}
\end{equation}
where by $\alpha^{\epsilon}, \beta^{\epsilon}$ and $\gamma^{\epsilon}$ we denote the dual variables. The Benders cut obtained from the Lagrangian dual of \eqref{subproblem_perturbed} is
    \begin{equation*}
        \phi^{\prime}(\bar y^{\epsilon}, \bar v^{\epsilon}) + \sum_{i \in I} \bar r_{y_i}^{\epsilon} (y_i - \bar y^{\epsilon}_{i}) +  \sum_{i \in I} 
 \bar r_{v_i}^{\epsilon}(v_i - \bar v^{\epsilon}_{i}) \geq D
    \end{equation*}
where $\bar r_{y}^{\epsilon}$ and $\bar r_{v}^{\epsilon}$ are the vectors of reduced costs associated with $y$ and $v$. We now state some properties of the perturbed problem and clarify the relationship between the perturbed and the non-perturbed problem.

\begin{lemma}\label{lemma_epsilon_technique} Let $(\bar r_y, \bar r_v)$ and $(\bar r^{\epsilon}_y, \bar r^{\epsilon}_v)$ be the reduced costs introduced for the Benders subproblem \eqref{subproblem_fixed} and its $\epsilon$-perturbed counterpart \eqref{subproblem_perturbed}. The following equations are valid
$$\bar r_{y} = \bar \beta, \quad \bar r_{v} = \bar \gamma  , \quad  \bar r^{\epsilon}_{y} = \bar \beta^\epsilon, \quad \bar r_{v}^\epsilon = \bar \gamma^\epsilon.$$
\end{lemma}
\begin{proof}
    Consider the non-perturbed problem \eqref{subproblem_fixed}. Since \eqref{subproblem_fixed} is an LP, from the equivalence between LP duality and Lagrangian duality for LP problems we have
    $$\sum_{j \in J} \bar \alpha_j + \sum_{i \in I} \bar \beta_i y_i  + \sum_{i \in I} \bar \gamma_i v_i =  \phi^{\prime}(\bar y, \bar v) + \sum_{i \in I} \bar r_{y_i} (y_i - \bar y_i) +  \sum_{i \in I} 
 \bar r_{v_i}(v_i - \bar v_i)$$
which implies
 $$\sum_{i \in I}  \bar \beta_i y_i  = \sum_{i \in I} \bar r_{y_i} y_i , \quad  \sum_{i \in I}  \bar \gamma_i v_i = \sum_{i \in I} \bar r_{v_i} v_i, \quad \sum_{j \in J} \bar \alpha_j  = \phi^{\prime}(\bar y, \bar v) - \sum_{i \in I} \bar r_{y_i} \bar y_i - \sum_{i \in I} \bar r_{v_i} \bar v_i,$$
meaning that $\bar r_{y_i} = \bar \beta_i$ and $\bar r_{v_i} = \bar \gamma_i$ for each $i \in I$, namely the dual variables associated with the fixing constraints are the reduced costs associated with the fixed variables. By applying the same procedure to the perturbed problem \eqref{subproblem_perturbed}, we get $\bar r^{\epsilon}_{y_i} = \bar \beta_i^\epsilon$ and $\bar r_{v_i}^\epsilon = \bar \gamma_i^\epsilon$ for each $i \in I$.
\end{proof}
\begin{flushright}
$\square$
\end{flushright}

\begin{proposition}\label{proposition_epsilon_technique}
If there exists a sufficiently small $\epsilon >0$ such that the dual solution $(\alpha^{\epsilon}, \beta^{\epsilon}, \gamma^{\epsilon})$ of the perturbed problem \eqref{subproblem_perturbed} is an optimal solution of the dual of the non-perturbed problem \eqref{subproblem_fixed}, then \begin{itemize}
    \item[(1)]  $\phi^{\prime}(\bar y, \bar v) =\displaystyle \phi^{\prime}(\bar y^{\epsilon}, \bar v^{\epsilon}) - \epsilon \sum_{i \in I:\, \bar y_i = 0} \bar r_{y_i}^{\epsilon} - \epsilon \sum_{i \in I:\, \bar v_i = 0} \bar r_{v_i}^{\epsilon}$, and
    \item[(2)] the following inequality
    \begin{equation}\label{bcut_perturbed}
        \phi^{\prime}(\bar y, \bar v) + \sum_{i \in I} \bar r_{y_i}^{\epsilon} (y_i - \bar y_i) + \sum_{i \in I} \bar r_{v_i}^{\epsilon} (v_i - \bar v_i) \geq D
    \end{equation}
   is a valid Benders cut for the non-perturbed problem \eqref{subproblem_fixed}.
\end{itemize}
\end{proposition}
\begin{proof}
    From the definition of $y^{\epsilon}$ and $v^{\epsilon}$ we have
\begin{equation}\label{def_y_eps}
    \displaystyle\sum_{i \in I} \bar y^{\epsilon}_{i} = \sum_{i \in I} \bar y_i + \epsilon \sum_{i \in I:\, \bar y_i = 0} 1
\end{equation}
\begin{equation}\label{def_v_eps}
    \displaystyle\sum_{i \in I} \bar v^{\epsilon}_{i} = \sum_{i \in I} \bar v_i + \epsilon \sum_{i \in I:\, \bar v_i = 0} 1.
\end{equation}
Consider the expression of the dual objective of the perturbed problem \eqref{subproblem_perturbed} at optimal value
$$\phi^{\prime}(\bar y^{\epsilon}, \bar v^{\epsilon}) = \sum_{j \in J} \bar\alpha_j^{\epsilon} + \sum_{i \in I} \bar y_i^{\epsilon} \bar\beta_i^{\epsilon} + \sum_{i \in I} \bar v_i^{\epsilon} \bar \gamma_i^{\epsilon}.$$
By using \eqref{def_y_eps}-\eqref{def_v_eps} we get
$$\displaystyle\phi^{\prime}(\bar y^{\epsilon}, \bar v^{\epsilon}) = \underbrace{\sum_{j \in J} \bar\alpha_j^{\epsilon} + \sum_{i \in I} \bar y_i \bar\beta_i^{\epsilon} + \sum_{i \in I} \bar v_i \bar \gamma_i^{\epsilon}}_{\phi^{\prime}(\bar y, \bar v)}  + \epsilon \sum_{i \in I:\, \bar y_i = 0} \overbrace{\bar \beta_i^{\epsilon}}^{\bar r_{y_i}^{\epsilon}} + \epsilon\sum_{i \in I:\, \bar v_i = 0} \overbrace{\bar \gamma_i^{\epsilon}}^{\bar r_{v_i}^{\epsilon}}$$
which proves \textit{(1)} for the assumption we made on $\epsilon$ and using Lemma \ref{lemma_epsilon_technique}.

If $\epsilon$ is sufficiently small that $(\alpha^{\epsilon}, \beta^{\epsilon}, \gamma^{\epsilon})$ is an optimal solution of the dual of the non-perturbed problem \eqref{subproblem_fixed}, then for Lemma \ref{lemma_epsilon_technique}, we can use $\bar r^{\epsilon}_{y_i}$ and $\bar r^{\epsilon}_{v_i}$ in \eqref{bcut} and \textit{(2)} is proved.
\end{proof}
\begin{flushright}
$\square$
\end{flushright}

\section{Implementation Details}\label{ch_implementation}

Having a correct and efficient generation of the Benders cuts is a crucial point of Benders decomposition. In this section, we describe the 
steps for the design of a Benders decomposition, revealing all the implementation ingredients that play an important role in the design of an effective code. 
In particular, we start by illustrating how to implement in practice Benders decomposition in a MISOCP solver in Section \ref{BDec_implementation}. The discussion is based on the solver we used, namely Gurobi \cite{gurobi}; however, our approach can be easily extended to other solvers. Then, 
in Section \ref{rounding} we report all the rounding operations we made and the thresholds we used to mitigate numerical issues. 


\subsection{Implementing Benders decomposition in practice}
\label{BDec_implementation}

There are two ways of implementing Benders decomposition: (1) using a cutting plane procedure, also known as multi-tree approach, where each time a Benders cut is generated, the cut is included in the master problem and the latter is solved again as a MISOCP until optimality; or (2) using a branch-and-Benders-cut approach, also known as single-tree approach, in which a single enumeration tree for solving the MISOCP is created and Benders cuts are separated at the nodes as in a classical branch-and-cut procedure. The traditional approach is the multi-tree and it requires the solution of possibly several master problems to optimality. The single-tree approach instead requires the solution of only one master problem and has become very popular in the recent literature (see, e.g., \cite{fischetti2016benders,fischetti2017redesigning,ljubic2012exact}). However, its implementation was mainly applied in an MILP (rather than MISOCP) context. We describe here the implementation of the two approaches.

Note that we do not use any specialized solution method to solve the master problem or the subproblem as well. Indeed, using a general-purpose MISOCP solver benefits many advantages, such as a better warm-start mechanism and effective heuristics. This choice simplifies the implementation process considerably. 

\subsubsection{Implementation of single-tree approach}

For the single-tree approach, the common procedure when using a MIP solver is to define a customized callback function to be called at every node of the branch-and-cut applied to the master problem. The callback function, described in Algorithm \ref{callback}, defines how the separation is done. First, at each node of the branching tree the continuous relaxation of the master problem is solved and either a fractional or an integer solution $y$ is obtained.
We restricted the generation of a separating cut to integer solutions only for all nodes except the root node. Specifically, at the root node we solve the subproblem to possibly generate Benders cuts at any solution, integral and fractional, of the master relaxation; at all the other nodes, we generate Benders cuts exclusively when the solution of the master relaxation is integer.   
If at the subproblem solution the target coverage is not reached, this solution must be discarded. Therefore, we generate a normalized Benders cut, inject this cut into the master problem and proceed through the branching tree until an optimal solution is found or the maximum time is exceeded.

\begin{algorithm}
\caption{Single-tree Benders implementation}
\label{callback}
\begin{algorithmic}
\If{(we are at the root node) OR (a MIP incumbent of the master problem is found)}
                \State extract the solution $(\bar y, \bar v)$ of the relaxed master problem
                \State update the values of the master variables $(y, v)$ in the subproblem constraints \eqref{fixing_y}-\eqref{fixing_v}
                \State solve the subproblem
                \If{(the target coverage is not reached)}
                \State get the reduced costs associated with the master variables in the subproblem
                \State add the normalized feasibility Benders cut
                \eqref{bcut} to the master problem
                \EndIf
 \EndIf
\end{algorithmic}
\end{algorithm}

We observe that each variable-fixing equation (i.e., \eqref{fixing_y}-\eqref{fixing_v}) is meant to be imposed by modifying the lower and upper bounds on the corresponding variable, e.g., for the master variable $y$ we can impose $\bar y \leq y \leq \bar y$, which can be handled very efficiently by the solver in a preprocessing phase when the $\bar y$ is given.

We set the parameter \textit{timeLimit} to stop the solution process when the maximum time is reached.
Moreover, in order to use a custom callback function, we set the optimization parameters \textit{LazyConstraints} to 1 and \textit{PreCrush} to 1 for the master problem. Note that not all commercial solvers allow you to apply a callback function to a problem that contains quadratic constraints.  

\subsubsection{Implementation of multi-tree approach}
The code of the multi-tree Benders decomposition method is reported in Algorithm \ref{while_loop}. It involves solving a master problem as an MISOCP iteratively, updating the subproblem at every optimal solution of the master, and generating Benders cuts until the master solution does not violate the required coverage (we denote it as the termination condition in the algorithm) or the process terminates due to exceeding of the time limit.

\begin{algorithm}
\caption{Multi-tree Benders implementation}
\label{while_loop}
\begin{algorithmic}
    \While{termination condition is not met or a maximum time is not exceeded}
         \State solve the relaxed master problem as an MISOCP
        \State extract the optimal solution $(\bar y, \bar v)$ of the master problem 
        \State update the values of the master variables $(y, v)$ in the subproblem constraints \eqref{fixing_y}-\eqref{fixing_v}
        \State solve the subproblem
        \If{the target coverage is not reached}
        \State get the reduced costs associated with the master variables in the subproblem
        \State generate the normalized feasibility Benders cut 
        \State add it as a linear constraint to the master problem
        \Else
            \State termination condition is met
        \EndIf
    \EndWhile
\end{algorithmic}
\end{algorithm}

Also in this case, each variable-fixing equation (i.e., \eqref{fixing_y}-\eqref{fixing_v}) is imposed by modifying the lower and upper bounds on the corresponding variable. 

\subsection{Thresholds and rounding operations}
\label{rounding}

We now describe all the rounding operations we made and the thresholds we used to deal with the numerical issues. 
When we get a MIP incumbent of the master in the single-tree approach or an optimal solution of the master in the multi-tree approach, we 
round to 0-1 the binary variables belonging to the master solution. This is done even if the values should already be 0-1, as the solver may lack precision. As for the fractional master solutions, we round to 0 all the negative values of y and to 1 all the values of y exceeding 1 for the same reason. 

When applying the $\epsilon$-procedure, instead of replacing zero coefficients with $\epsilon$, we replace the coefficients below a given small tolerance $\text{tol}_{\beta} >0$ with $\epsilon$. This is because the solver is supposed to treat values below some internal threshold as zeros.

Once the subproblem at a given master solution is solved, the condition we should theoretically check to decide on the generation of the Benders cut is if $\phi^{\prime}(\bar y, \bar v) < D$.
To get a more precise check of this condition from the numerical point of view, we implement this condition as $\phi^{\prime}(\bar y, \bar v)/D < 1 - \text{tol}_{\alpha}$, where $\text{tol}_{\alpha}>0$ is a given small tolerance taking two different values, one for integer master solutions and another for fractional master solutions. Indeed, for a correct implementation of the Benders decomposition we only need to separate integral points. However,
generating cuts also at fractional solutions at the root node improves the quality of bounds and reduces the size of the search space. Therefore, we decided to separate also fractional solutions at the root node. Since this could sometimes slow down the performance, to keep under control the number of Benders user cuts generated, we set a higher value for the tolerance tol$_{\alpha}$, implying that their generation is performed only if the coverage achieved at the fractional point is far from the coverage required.

\section{Computational Experiments}
\label{ch_results}

All the experiments run on a Ubuntu server with an Intel(R) Xeon(R) Gold 5218 CPU running at 2.30 GHz, with 96 GB of RAM and 8 cores. As for the optimizer, we use Gurobi 10.0.3 with a time limit (TL) of 900 seconds. On Gurobi, to guarantee better numerical precision, we set integrality tolerance \textit{IntFeasTol} to $10^{-9}$, feasibility tolerance \textit{FeasibilityTol} to $10^{-9}$, the optimality tolerances \textit{OptimalityTol} to $10^{-9}$ and \textit{MIPGap} to 0. 
As for the setting of Benders decomposition methods in Gurobi, we enable the parameters for lazy constraints (\textit{LazyConstraints}=1) and user cuts (\textit{PreCrush}=1) for the single-tree approach to avoid any reductions and transformations that are incompatible with a custom callback, and turn off presolve (\textit{Presolve=0}) in all the Benders approaches since during the presolve phase the model is not complete as many constraints are missing and the optimizer could make reductions without knowing the whole problem. We also set the following values for the tolerances defined in the previous section: $\text{tol}_{\alpha} = 10^{-6}$ for integral points and $\text{tol}_{\alpha} = 0.5$ for fractional points, $\text{tol}_{\beta}=10^{-8}$, $\epsilon=10^{-8}$. We also turn off presolve and use a dummy lazy callback on the MISOCP model solved directly with Gurobi to have a fair comparison with our Benders methods.
To have a fairer comparison with our Benders, we turn off presolve and use a dummy callback (i.e., an empty function declared as a callback) to solve the MISOCP model via Gurobi blackbox. The use of a dummy callback is typical in the MILP context to disable dynamic search and other advanced features that solvers can apply when given a full model versus models that require lazy cut separation.

The testbed is made of large-size instances derived from Cordeau et al. \cite{cordeau2019benders}.
The original instances are tailored for realistic scenarios where the number of customers is much larger than the number of potential facility locations (i.e., $|J| >> |I|$). In particular, we selected instances with 1000 customers and 100 potential facility locations. Customer demand in these instances was generated by Cordeau et al. by uniformly sampling from the range [1,100] and rounding to the nearest integer. The spatial coordinates for both customers and facilities were randomly selected from the interval [0, 30].

We further adapted the testbed to suit our study on robust and congested settings by injecting uncertainty in the customer demand and generating congestion quadratic cost. Specifically, for each customer $j$, the deviation $\hat d_j$ from the nominal demand $d_j$ was set as a random integer drawn from the interval [$0, 20\%\, d_j$], meaning that the maximum possible deviation from the nominal value is 20\%. Customers affected by uncertainty are those for which $\hat d_j > 0$. Protection from the uncertainty in customer demand was introduced with varying levels of $\Gamma$. Congestion cost was computed by setting $a_i = a$ and $b_i = b$ for each $i \in I$: the values of $a$ and $b$ are chosen in such a way as to balance linear and quadratic costs, and obtain optimal solutions requiring the opening of less than half of the potential facilities.

\subsection{Sensitivity analysis}

In this paper we are studying a robust problem characterized by an unusual feature: the worst realizations occur in the case of increasing demand for what concerns the congestion, and in the opposite case of decreasing demand for the coverage. Therefore, we find it interesting to investigate how the optimal value, the coverage reached, the facility loads and the number of open facilities vary based on the values of the protection level $\Gamma$ in each of the following cases: the deterministic case, protection against load uncertainty, protection against coverage uncertainty, and protection against both load and coverage uncertainties. 

The primary goal of this sensitivity analysis is to understand how changes in the parameter $\Gamma$ influence the optimal solution of our problem. 
Therefore, we solve eight different instances using Gurobi on the perspective MISOCP reformulation of the problem, incrementally increasing the value of $\Gamma$ in $\{0, 2.5\%, 5\%, 10\%, 15\%, 20\%, \\
30\%, 40\%,$ $50\%, 60\%, 70\%, 80\%, 90\%, 100\%\}|J|$ and observing the resulting effects on specific metrics. The instances are characterized by a different coverage radius (R) and seed (s) used to generate the instance. 
In particular, for each instance, starting from $\Gamma$ equal to zero, namely the deterministic setting where we consider no protection towards the uncertainty, until $\Gamma$ equal to the number of customers, namely the case with the maximum protection towards the uncertainty, we report in Tables \ref{tab:sensitivity_1}- \ref{tab:sensitivity_2} the values of the overall cost (ObjVal), the number of open facilities (\#Fac), the facility opening cost (OpenCost), the congestion cost (CongCost), the average load (Load) and the coverage reached (Cov) at the optimal solution. Additionally, Figures \ref{fig:sensitivity_objval}-\ref{fig:sensitivity_load_coverage} graphically represent these metrics as functions of $\Gamma$ for a representative instance. 

From the results collected, we can claim that for increasing values of $\Gamma$, we get increasing optimal values due to an increasing number of open facilities (hence increasing opening cost) and an increasing average load at the facilities (hence increasing congestion cost), until a certain value of $\Gamma$ denoted in the figures as $\Gamma^{\star}$. For $\Gamma \geq \Gamma^{\star}$, the values of these metrics remain stable, meaning that over-conservative protection is not necessary as it does not affect the optimal value. We can observe that $\Gamma^{\star}$ takes smaller values in the case we only protect from the load uncertainty, higher values in the case we only protect from the coverage uncertainty, intermediate values (around $50\%|J|$) in the case we protect from both. Also note how the optimal objective value in the double protection scenarios is closer to that of coverage-only protection scenarios. 
As for the coverage reached, it remains constant at the nominal coverage for the case of load uncertainty only, and it is increasing with increasing values of $\Gamma$ for the other cases. We remark that we are empirically overestimating $\Gamma$ (and $\Gamma^{\star}$ as well) as we do not consider all possible values of $\Gamma$.

Our findings indicate that the optimal solutions are highly sensitive to the value of the protection level $\Gamma$, especially if we consider an uncertain coverage. This understanding is crucial for making informed planning decisions and assessing the reliability of a deterministic solution.

\begin{table}[h!]
\caption{Sensitivity analysis (part I). We denote by * the optimal value in correspondence to $\Gamma^{\star}$.}\label{tab:sensitivity_1}
\resizebox{\textwidth}{!}{
\begin{tabular}{llrrrrrrrrrrrrrrrrrr}
\multicolumn{1}{l}{}       &                                    & \multicolumn{6}{c}{\textbf{Uncertainty in Load \& Coverage}}                                                                                                                                                                                             & \multicolumn{6}{c}{\textbf{Uncertainty in Coverage}}                                                                                                                                                                                                     & \multicolumn{6}{c}{\textbf{Uncertainty in Load}}                                                                                                                                                                                                         \\
\cmidrule(lr){3-8}\cmidrule(lr){9-14}\cmidrule(lr){15-20} 
\textbf{(s-R)}             & \multicolumn{1}{c}{\textbf{$\bm\Gamma$}} & \multicolumn{1}{c}{\textbf{ObjVal}} & \multicolumn{1}{c}{\textbf{\#Fac}} & \multicolumn{1}{c}{\textbf{OpenCost}} & \multicolumn{1}{c}{\textbf{CongCost}} & \multicolumn{1}{c}{\textbf{Load}} & \multicolumn{1}{c}{\textbf{Cov}} & \multicolumn{1}{c}{\textbf{ObjVal}} & \multicolumn{1}{c}{\textbf{\#Fac}} & \multicolumn{1}{c}{\textbf{OpenCost}} & \multicolumn{1}{c}{\textbf{CongCost}} & \multicolumn{1}{c}{\textbf{Load}} & \multicolumn{1}{c}{\textbf{Cov}} & \multicolumn{1}{c}{\textbf{ObjVal}} & \multicolumn{1}{c}{\textbf{\#Fac}} & \multicolumn{1}{c}{\textbf{OpenCost}} & \multicolumn{1}{c}{\textbf{CongCost}} & \multicolumn{1}{c}{\textbf{Load}} & \multicolumn{1}{c}{\textbf{Cov}} \\
\end{tabular}}
\end{table}
\begin{table}[h!]
\caption{Sensitivity analysis (part II). We denote by * the optimal value in correspondence to $\Gamma^{\star}$.}\label{tab:sensitivity_2}
\resizebox{\textwidth}{!}{
\begin{tabular}{llrrrrrrrrrrrrrrrrrr}
\multicolumn{1}{l}{}       &                                    & \multicolumn{6}{c}{\textbf{Uncertainty in Load \& Coverage}}                                                                                                                                                                                             & \multicolumn{6}{c}{\textbf{Uncertainty in Coverage}}                                                                                                                                                                                                     & \multicolumn{6}{c}{\textbf{Uncertainty in Load}}                                                                                                                                                                                                         \\
\cmidrule(lr){3-8}\cmidrule(lr){9-14}\cmidrule(lr){15-20} 
\textbf{(s-R)}             & \multicolumn{1}{c}{\textbf{$\bm\Gamma$}} & \multicolumn{1}{c}{\textbf{ObjVal}} & \multicolumn{1}{c}{\textbf{\#Fac}} & \multicolumn{1}{c}{\textbf{OpenCost}} & \multicolumn{1}{c}{\textbf{CongCost}} & \multicolumn{1}{c}{\textbf{Load}} & \multicolumn{1}{c}{\textbf{Cov}} & \multicolumn{1}{c}{\textbf{ObjVal}} & \multicolumn{1}{c}{\textbf{\#Fac}} & \multicolumn{1}{c}{\textbf{OpenCost}} & \multicolumn{1}{c}{\textbf{CongCost}} & \multicolumn{1}{c}{\textbf{Load}} & \multicolumn{1}{c}{\textbf{Cov}} & \multicolumn{1}{c}{\textbf{ObjVal}} & \multicolumn{1}{c}{\textbf{\#Fac}} & \multicolumn{1}{c}{\textbf{OpenCost}} & \multicolumn{1}{c}{\textbf{CongCost}} & \multicolumn{1}{c}{\textbf{Load}} & \multicolumn{1}{c}{\textbf{Cov}} \\
\end{tabular}}
\end{table}

\begin{figure}[h!]
    \centering
    \includegraphics[width=9cm]{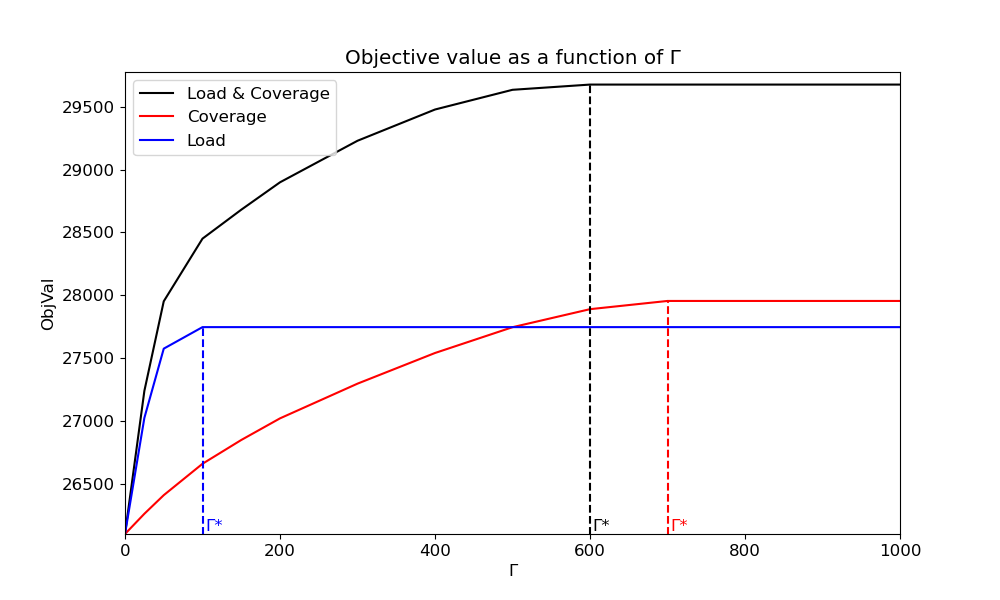}\includegraphics[width=9cm]{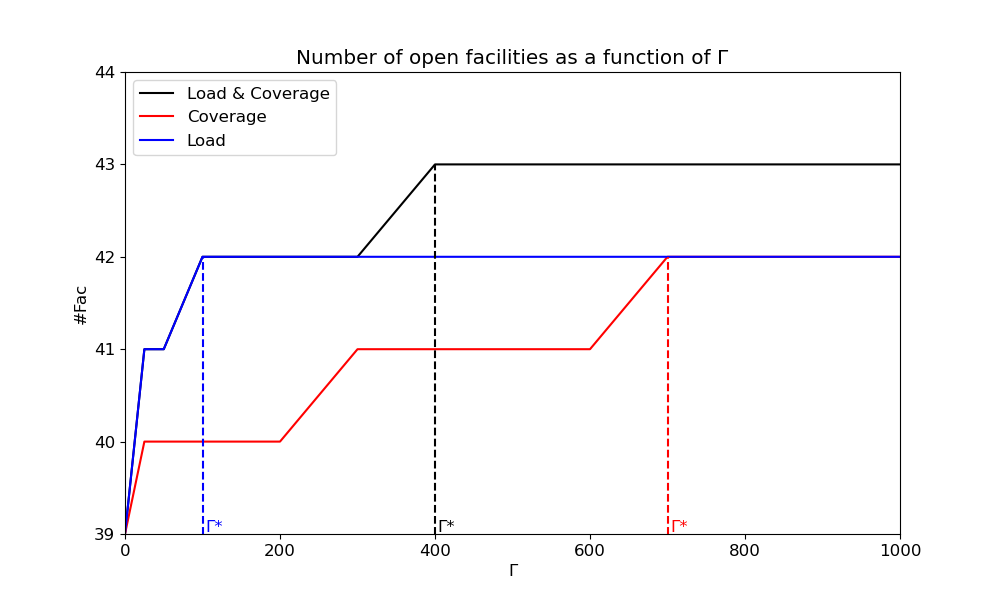}
    \caption[Solution profiles]{Plot of the variation of the overall objective value (on the left) and the number of open facilities (on the right)  based on $\Gamma$ values for an instance having 1000 customers, seed 1 and coverage radius 5.5.}
    \label{fig:sensitivity_objval}
\end{figure}
\begin{figure}[h!]
    \centering
    \includegraphics[width=9cm]{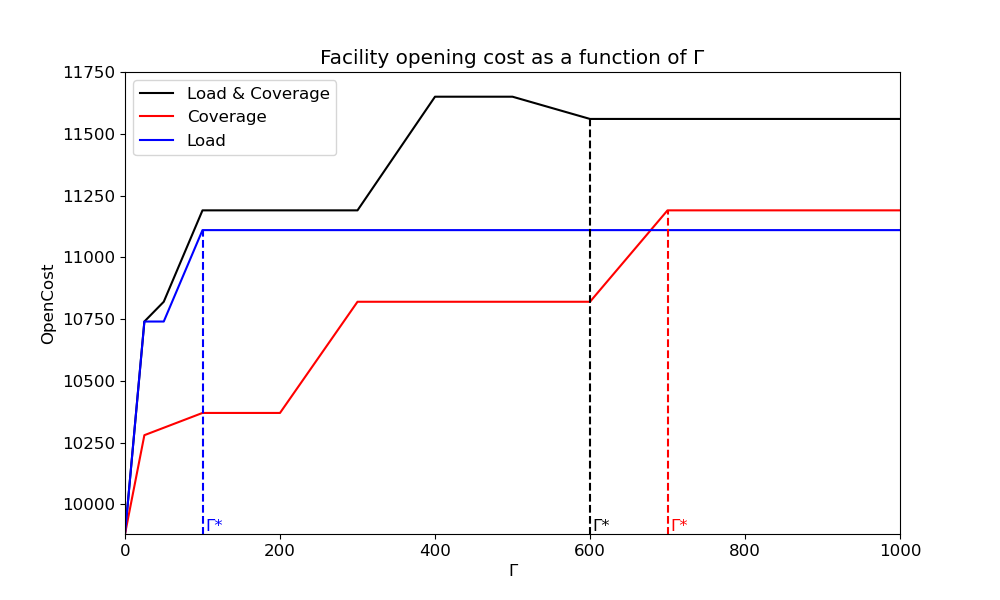}\includegraphics[width=9cm]{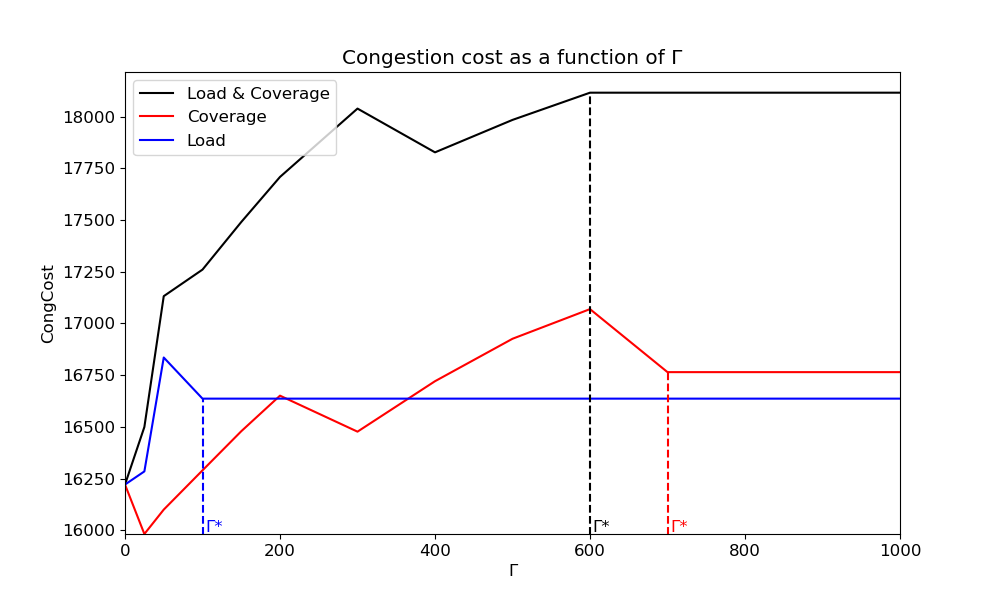}
    \caption[Solution profiles]{Plot of the variation of the facility opening cost (on the left) and the congestion cost (on the right) based on $\Gamma$ values for an instance having 1000 customers, seed 1 and coverage radius 5.5.}
    \label{fig:sensitivity_opening_congestion}
\end{figure}
\begin{figure}[h!]
    \centering
    \includegraphics[width=9cm]{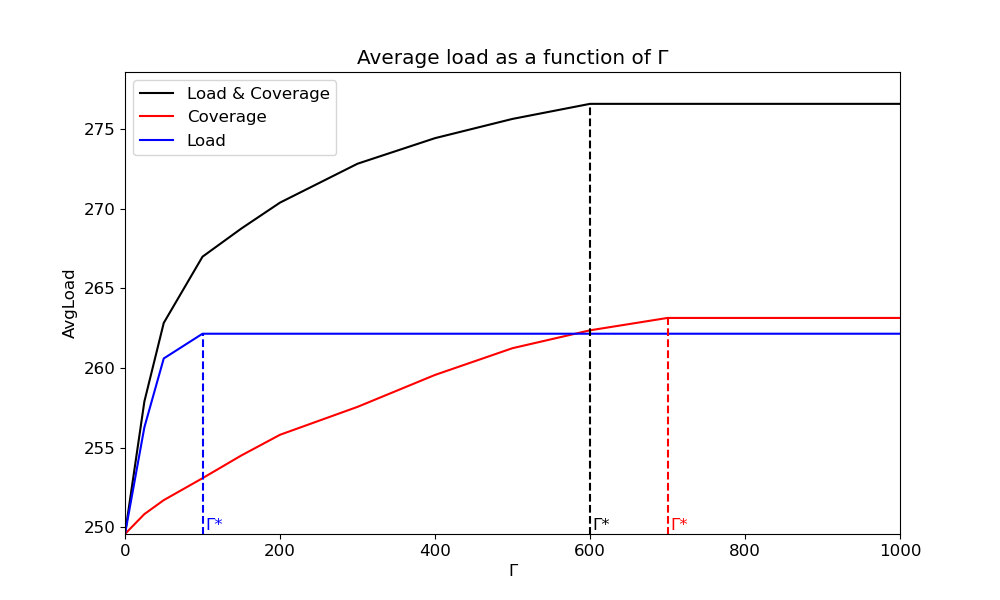}\includegraphics[width=9cm]{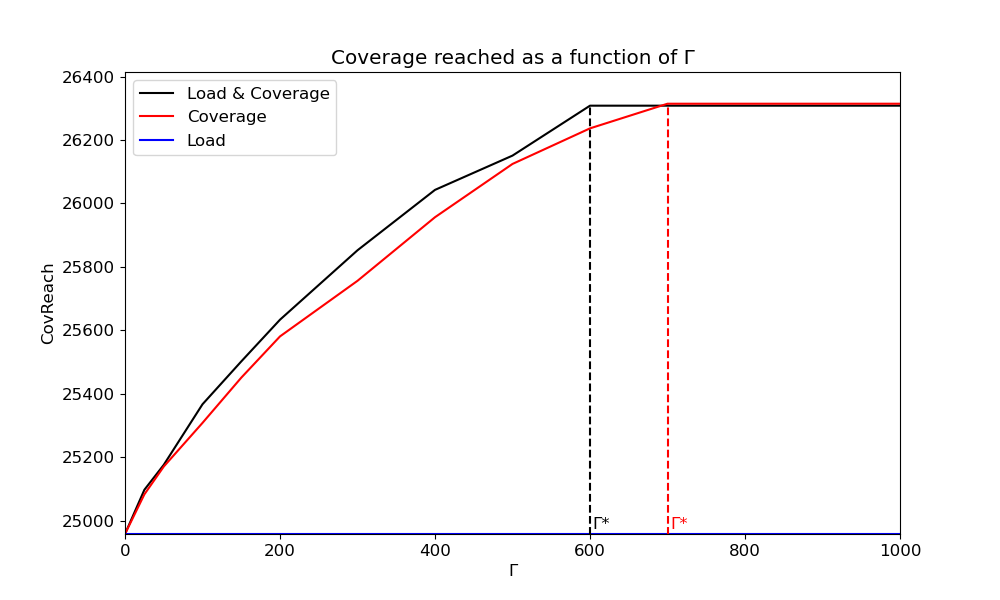}
    \caption[Solution profiles]{Plot of the variation of the average load (on the left) and the coverage reached (on the right) based on $\Gamma$ values for an instance having 1000 customers, seed 1 and coverage radius 5.5.}
    \label{fig:sensitivity_load_coverage}
\end{figure}

\subsection{Performance analysis}
We use a set of 128 instances of the CPSCLP to compare the performance of different methods and formulations introduced in this paper. The characteristics of each instance, including its optimal value and the number of open facilities at the optimal solution, are reported in Table \ref{tab:instances} of the Appendix \ref{sec:appendix_tables}. We use the following values for the protection level $\Gamma$: $2.5\%|J|, 5\%|J|, 10\%|J|, 15\%|J|, 20\%|J|, 30\%|J|, 40\%|J|, 50\%|J|$. We limit $\Gamma$ at 50\% of the number of customers to avoid overly conservative models.

The purpose of this section is to evaluate the performance of our Benders approaches as possible competitive exact algorithms to solve robust instances of the CPSCLP. To the best of our knowledge, no previous computational study appeared in the literature for such a problem, despite its potential theoretical and practical relevance. For this reason, in this section we compare the performance of our Benders algorithms against the direct use of Gurobi applied as a black-box MISOCP solver to the perspective reformulation of the robust counterpart of the problem. We also test Gurobi on the extended MIQP formulation of the robust counterpart of the problem to show the effectiveness of the perspective reformulation. We choose to compare our Benders algorithms to Gurobi as it is considered one of the state-of-the-art MIP solvers and is usually used as benchmark to compare the performance of newly developed exact algorithms. Moreover, Gurobi allows the implementation of callbacks on problems with quadratic constraints, a necessary feature for the implementation of our single-tree approaches.

We discuss the results given by six exact algorithms: Gurobi on the extended MIQP formulation \eqref{whole_extended} (denoted with MIQP from now on), Gurobi on the perspective MISOCP reformulation \eqref{whole_perspective} (denoted with MISOCP from now on), and four versions of our Benders algorithms directly on the perspective reformulation: the single-tree version (ST-BEN) and the single-tree version using the $\epsilon$-technique (ST$\epsilon$-BEN), the multi-tree version (MT-BEN) and the multi-tree version using the $\epsilon$-technique (MT$\epsilon$-BEN). Numerical results are reported in Tables \ref{tab:results_CPSCLP_1}-\ref{tab:results_CPSCLP_3} of the Appendix \ref{sec:appendix_tables}. Using the values reported in these tables, we made a graphical representation of the overall performance of the different exact algorithms in Figure \ref{fig:sol_profile}, illustrating the percentage of problems solved by each algorithm as a function of the computational time. It is important to note that these solution profiles do not represent the cumulative solution time but show the percentage of problems the algorithm can solve within a certain amount of time. The best performance are graphically represented by the curves in the upper part of the plot. From this representation, we can also infer the percentage of instances solved by each algorithm.

\begin{figure}[h!]
    \centering
    \includegraphics[width=10cm]{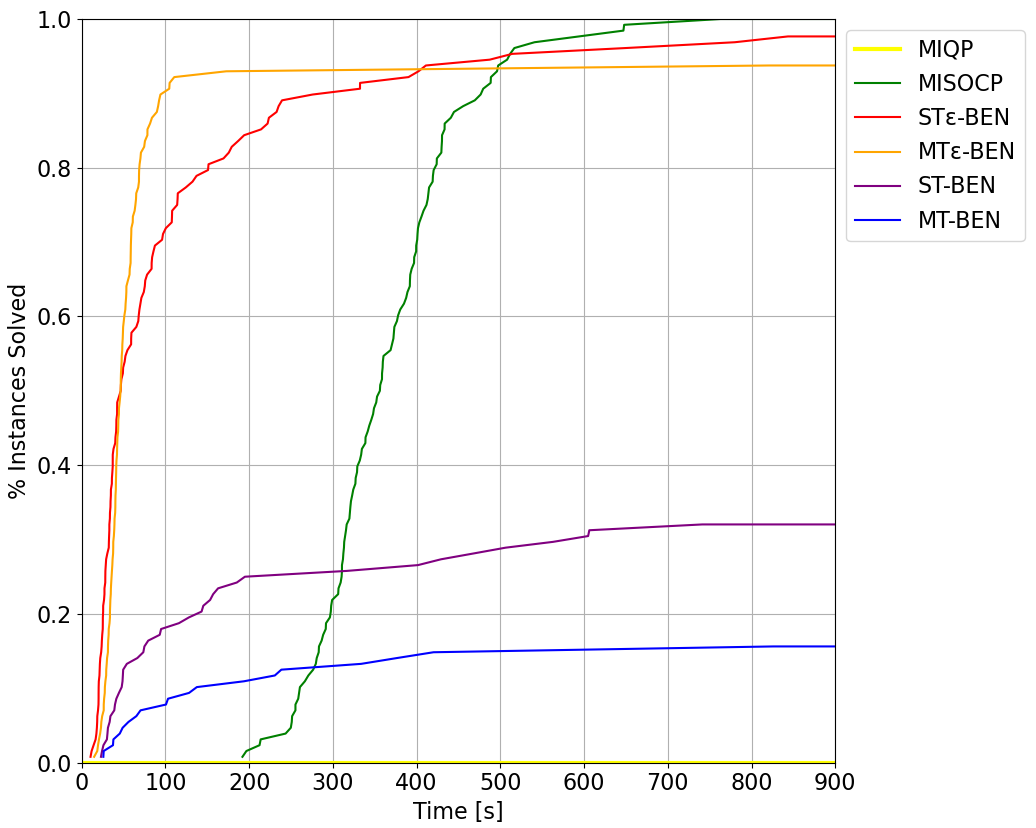}
    \caption[Solution profiles]{Solution profiles considering the whole benchmark set of 128 instances.}
    \label{fig:sol_profile}
\end{figure}

In particular, the profiles show that the MIQP model behaves poorly for this problem as none of the instances can be solved by this method and the relative gaps at the end of the optimization are pretty high (all over 50\%). The second and third worst performing methods for what concerns the number of solved instances are the Benders approaches without the $\epsilon$-technique. Indeed, 
ST-BEN solved 41 instances (around 32\%) and MT-BEN solved 20 instances (around 16\%). The Benders methods with $\epsilon$-technique, instead, solved most of the instances: ST$\epsilon$-BEN solved 125 instances (around 98\%) and MT$\epsilon$-BEN solved 120 instances (around 94\%). Finally, MISOCP solved all 128 instances.

As for resolution times, the algorithms running on the perspective reformulation are much more faster than the one running on the extended formulation: indeed, we cannot solve any instance within the time limit using the MIQP model. The fastest methods are ST$\epsilon$-BEN and MT$\epsilon$-BEN in a large percentage of the instances considered (around 97\%), and MISOCP in the remaining ones. With ST$\epsilon$-BEN, times of MISOCP are reduced on average by almost 67\%, with MT$\epsilon$-BEN, times of MISOCP are reduced on average by almost 75\%. We also remark that the two Benders methods using the $\epsilon$-technique converge to the same optimal solution of MISOCP for the chosen value of $\epsilon$.

For what concerns the direct comparison between ST$\epsilon$-BEN and MT$\epsilon$-BEN, we can notice that both approaches are very effective, with MT$\epsilon$-BEN being slightly superior in terms of solution times, and ST$\epsilon$-BEN being slightly superior in terms of the number of solved instances. 

From Figure \ref{fig:boxplot_bcuts} we further compare the Benders approaches with respect to the number of cuts generated (on integral and fractional points for the single-tree approach, only at integral points for the multi-tree approach). 
From the comparison of the box plots we can assess that the generation of cuts is reduced in both single and multi-tree approaches when using the $\epsilon$-technique, especially for the single-tree. 
The number of Benders cuts generated on average by the multi-tree approach is reduced by the $\epsilon$-technique by almost 75\%, and in the single-tree approach by 83\%. Furthermore, using the $\epsilon$-technique the number of cuts generated remains stable, while when it is not used the number of cuts reaches high values (up to 522 cuts for ST-BEN) in several instances.

\begin{figure}[h!]
    \centering
    \includegraphics[width=10cm]{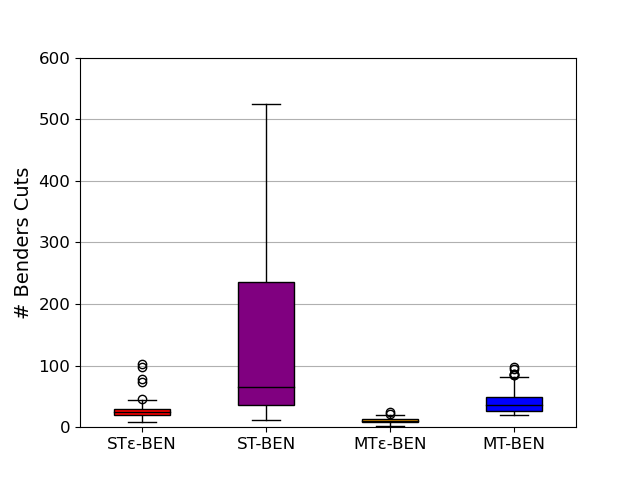}
    \caption[Boxplot of the number of Benders cuts]{Boxplot of the number of Benders cuts generated by Benders decomposition methods.} 
    \label{fig:boxplot_bcuts}
\end{figure}

In Figures \ref{fig:sol_profiles_gamma_1}-\ref{fig:sol_profiles_gamma_2} we report the solution profiles for each fixed value of $\Gamma$ to assess how $\Gamma$ affects the performance of each solution algorithm. 
\begin{figure}[h!]
    \centering
    \includegraphics[width=9cm]{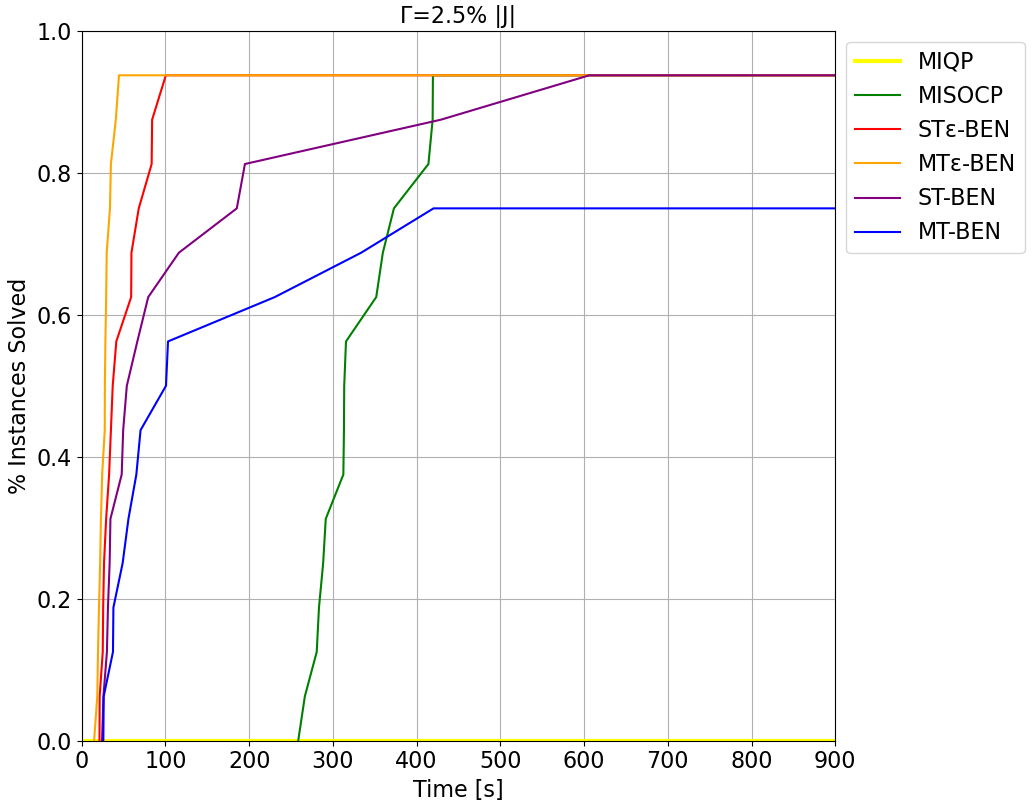}\includegraphics[width=9cm]{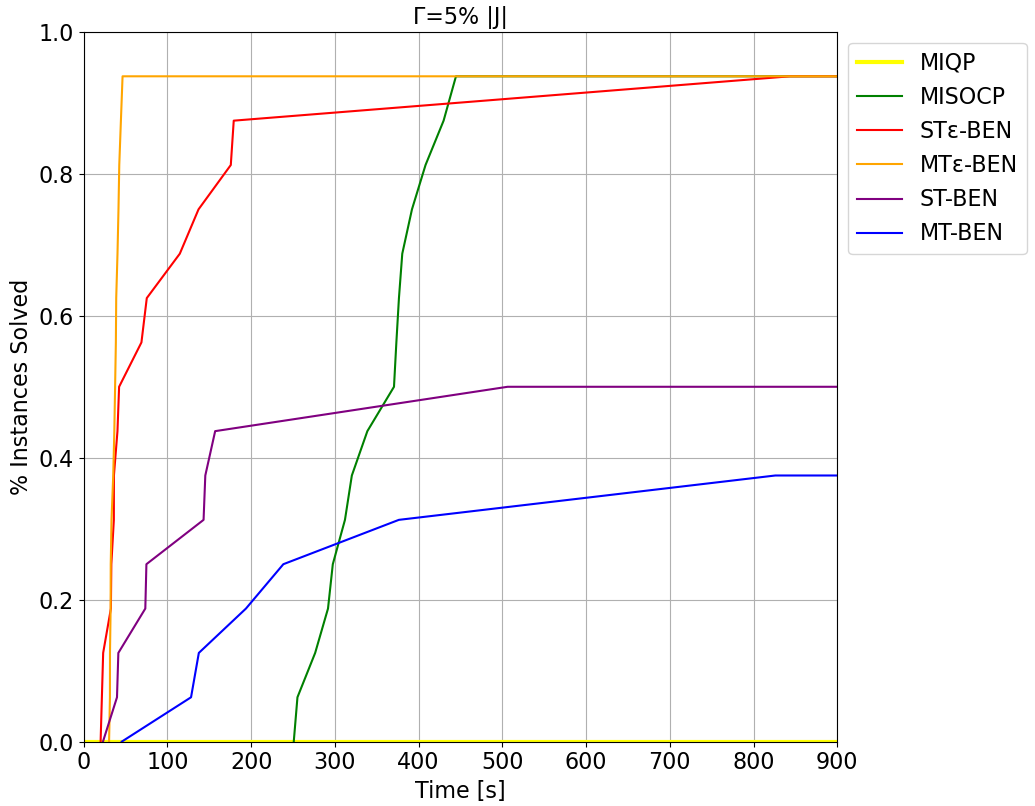}\\
    \includegraphics[width=9cm]{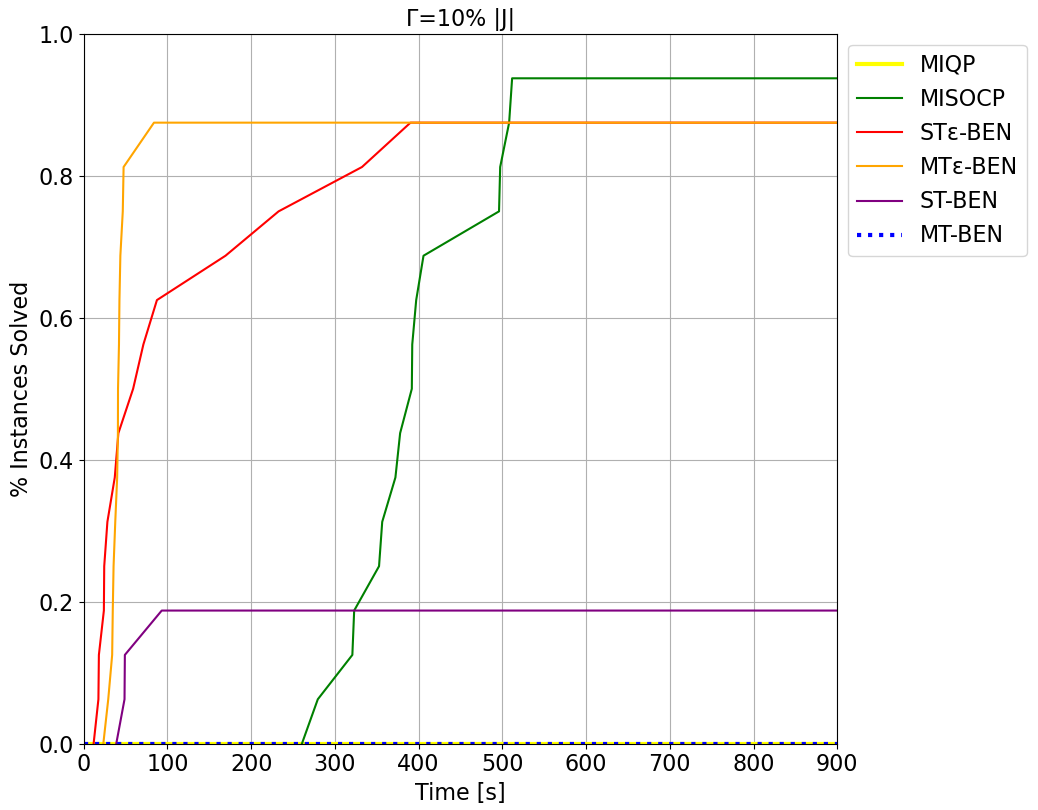}\includegraphics[width=9cm]{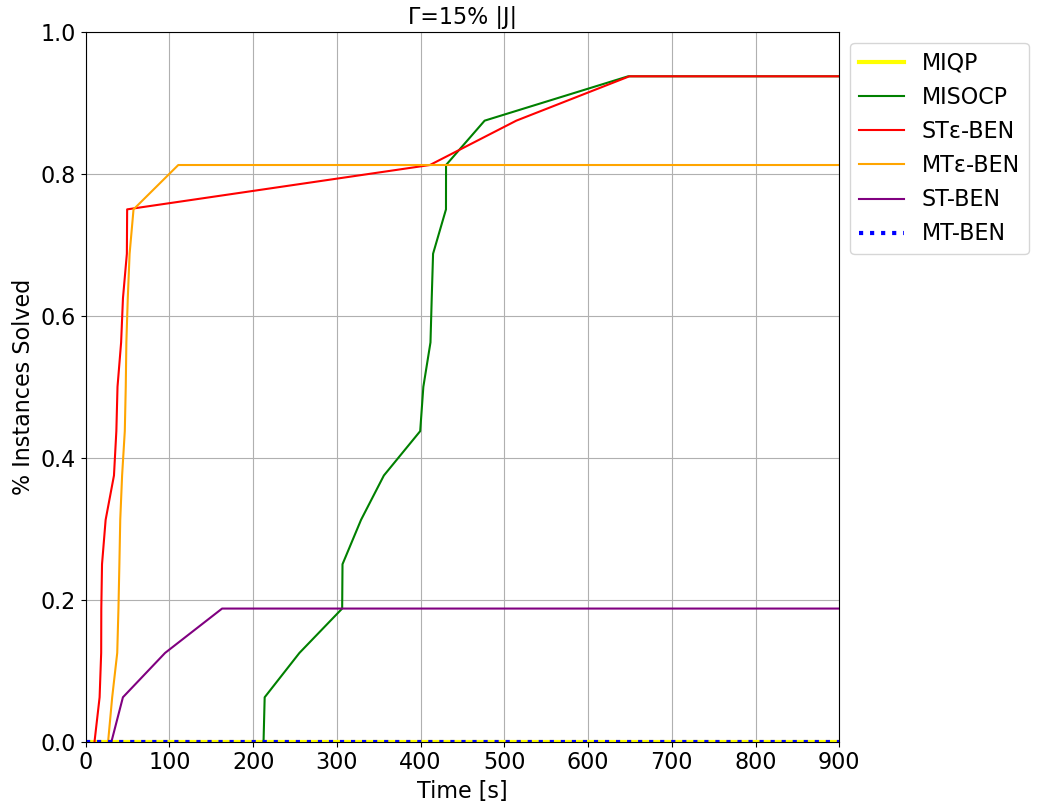}
    \caption[Solution profiles for low values of $\Gamma$]{Solution profiles for low values of $\Gamma$.}
    \label{fig:sol_profiles_gamma_1}
\end{figure}
\begin{figure}[h!]
    \centering
    \includegraphics[width=9cm]{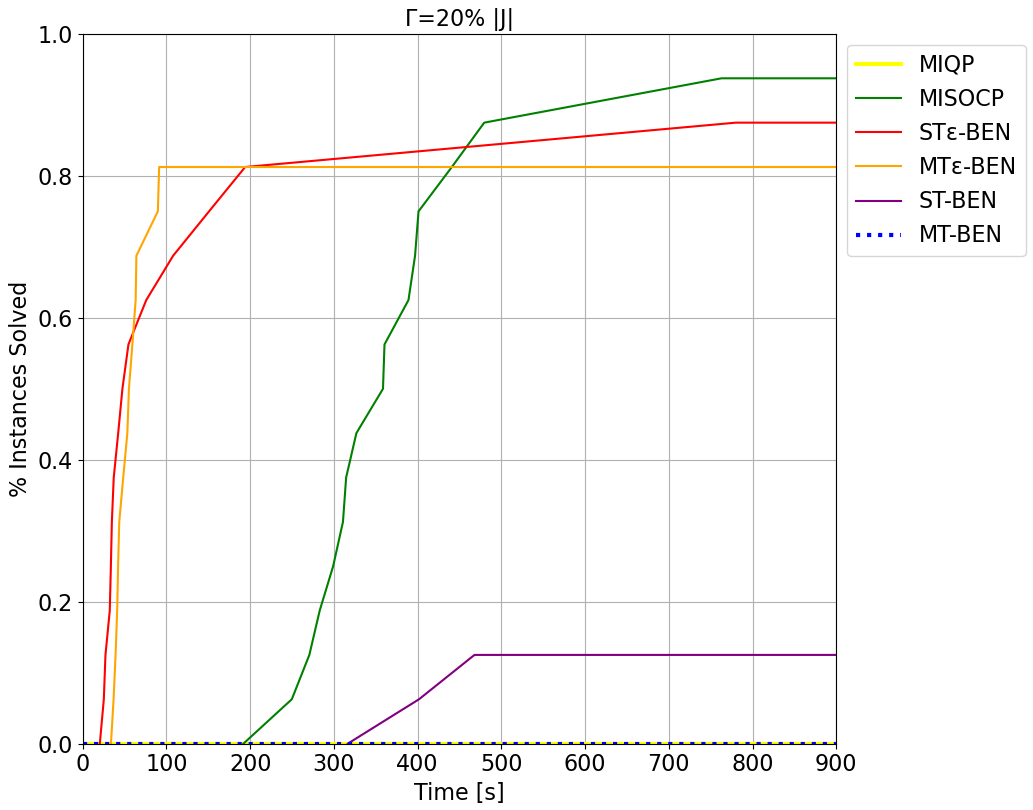}\includegraphics[width=9cm]{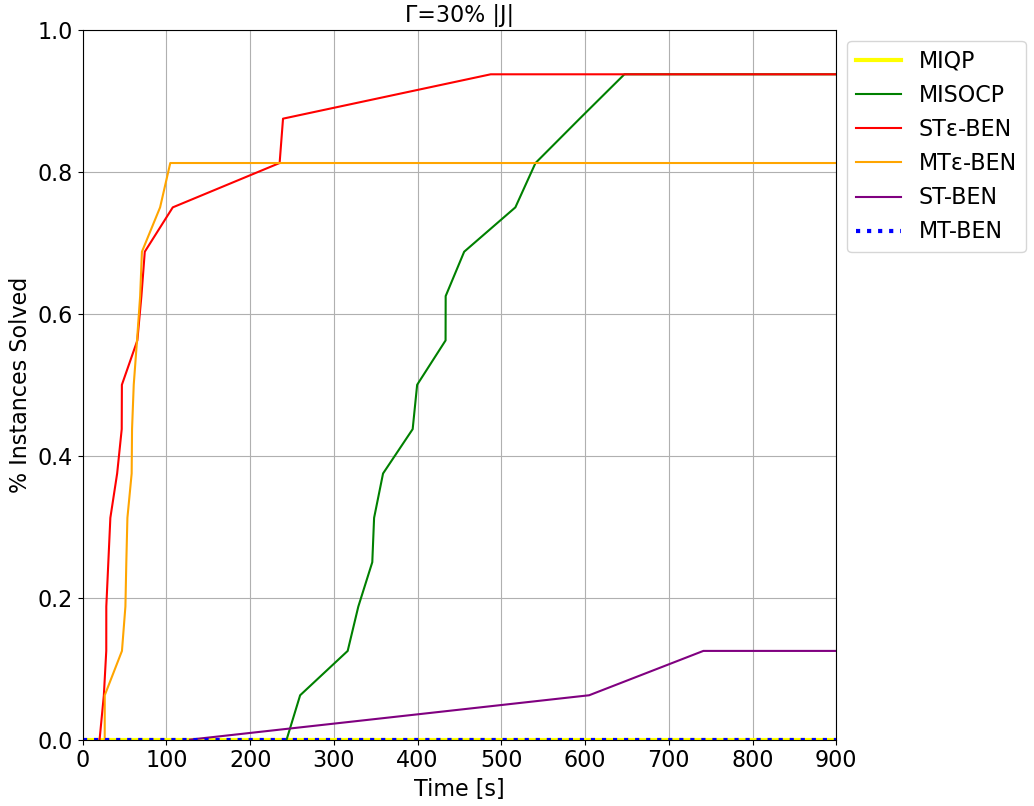}\\\includegraphics[width=9cm]{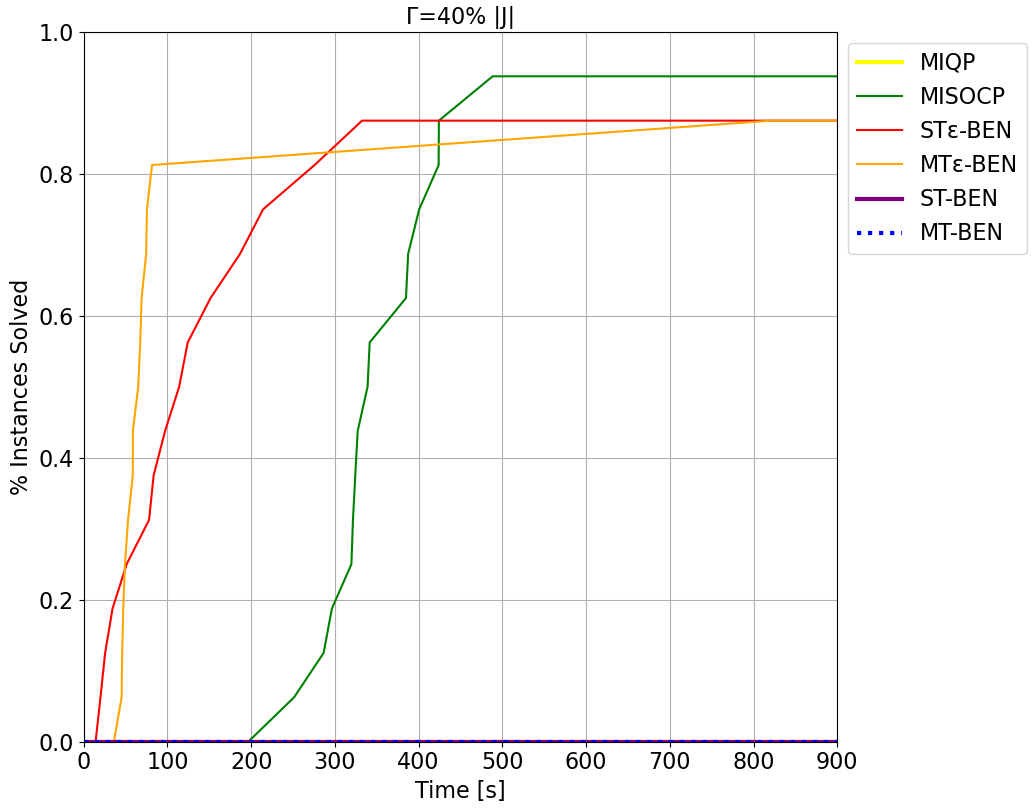}\includegraphics[width=9cm]{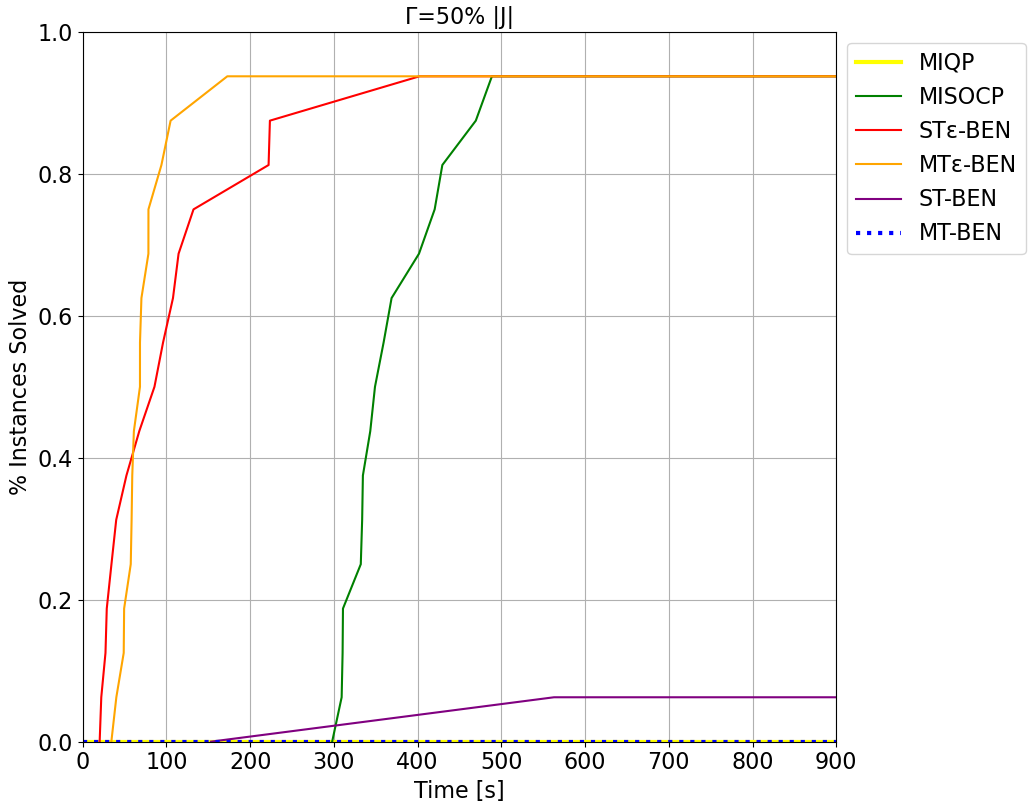}
    \caption[Solution profiles for medium to high values of $\Gamma$]{Solution profiles for medium to high values of $\Gamma$.}
    \label{fig:sol_profiles_gamma_2}
\end{figure}
From these figures, ST$\epsilon$-BEN and MT$\epsilon$-BEN consistently emerge as the best-performing methods across all tested values of $ \Gamma $. Specifically, MT$\epsilon$-BEN significantly outperforms ST$\epsilon$-BEN for very low values of $ \Gamma $ ($ \Gamma = \{2.5\%|J|, 5\%|J|\} $). 
As for Benders methods that do not employ the $\epsilon$-technique, they exhibit good performance at low $\Gamma$ values. Specifically, ST-BEN outperforms MISOCP in more than 80\% of instances when $\Gamma = 2.5\%|J|$ and in over 40\% of instances when $\Gamma = 5\%|J|$. Similarly, MT-BEN beats MISOCP in 70\% of cases having $\Gamma = 2.5\%|J|$ and in nearly 30\% when $\Gamma = 5\%|J|$. However, their performance drastically declines as $\Gamma$ increases, indicating a high limitation in their scalability with respect to increasing levels of protection against uncertainty.

\section{Conclusions}\label{ch_conclusions}
We presented a novel investigation into solving a robust and convex quadratic variant of the partial set covering location problem using modern Benders decomposition. 
Our new model achieves a better resource allocation and more balanced solutions to facility location problems affected by congestion and demand uncertainty. We investigated the case with demand uncertainty affecting both feasibility and optimality using $\Gamma$-robustness. We showed how a simple implementation of Benders decomposition is successful on a perspective reformulation of the robust counterpart of the problem, describing in detail all implementation ingredients we used to design an effective code
embedded in a modern MISOCP solver. Even though Benders decomposition is a traditional technique, the innovation of our proposal consists in being one of the first to use callbacks in combination with quadratic constraints, a new option provided by (few) state-of-the-art MISOCP solvers. Furthermore, we were able to demonstrate the effectiveness of Benders decomposition, even though the Benders subproblem is non-separable. 
We compared single and multi-tree implementations of Benders decomposition with the default branch-and-cut of a state-of-the-art MISOCP solver, also testing the use of a novel ad hoc cut-strengthening technique for degenerate subproblems. Using this technique we were able to reduce computational times of Gurobi on average by 75\% with the multi-tree and 67\% with the single-tree procedure on 128 large-size instances derived from the literature. We also conducted a sensitivity analysis to understand how changes in the protection level $\Gamma$ influence the optimal solutions of the instances. Our findings indicate that $\Gamma$ has a significant impact on the overall cost, the number of facilities to open, the coverage reached and the facility load, suggesting the unreliability of a deterministic setting of our problem and the importance of planning decisions that are robust towards the uncertainty. 

\bibliographystyle{abbrv}


\section*{Acknowledgements}
Alice Calamita acknowledges financial support from Progetto di Ricerca Medio Sapienza Uniroma1 - n. RM12117A857F64F2. 
Laura Palagi acknowledges financial support from Progetto di Ricerca Medio Sapienza Uniroma1 - n. RM1221816BAE8A79.

\appendix
\section{Appendix}
\label{sec:appendix_proof}

We provide the proof to the Theorem \ref{th:appendix}, following what is done in Theorem 1 of \cite{bertsimas2003robust}.


\begin{theorem}
    The extended formulation \eqref{whole_extended} is equivalent to \eqref{new_def_v}.
\end{theorem}

\begin{proof}
    
We show how to get formulation \eqref{whole_extended}.

Given $i \in I$ and a vector $x_i^{\star}\in\R^{|J|}$, we define the protection function $\alpha$ representing the sum of the $\Gamma$ largest deviations for $i \in I$ as
\begin{equation*}
    \alpha(x_i^{\star}) = \max_{\{S: \ S \subseteq J, |S| \leq \Gamma\}}\left\{\sum_{j \in S} \hat d_j x^{\star}_{ij} \right\}.
\end{equation*}
By introducing the binary variables $w_{ij}$, the protection function $\alpha$ can be equivalently formulated as
\begin{equation*}
\begin{array}{rlr}
    \alpha(x_i^{\star}) = \max & \displaystyle \sum_{j \in J} \hat d_j x^{\star}_{ij}w_{ij}\\
    s.t. & \displaystyle \sum_{j \in J} w_{ij} \leq \Gamma\\
    & w_{ij} \in \{0, 1\} & \quad j \in J
    \end{array}
\end{equation*}
which equals to
\begin{equation}\label{alpha_primal}
\begin{array}{rlr}
     \alpha(x_i^{\star}) = \max & \displaystyle \sum_{j \in J} \hat d_j x^{\star}_{ij}w_{ij}\\
    s.t. & \displaystyle \sum_{j \in J} w_{ij} \leq \Gamma\\ 
    & 0 \leq w_{ij} \leq 1 & \quad j \in J 
    \end{array}
\end{equation}
where we relaxed the integrality of variables $w_{ij}$ since the matrix of the constraints is totally unimodular and $\Gamma$ is integral (linear programs of this form have integral optima). 
Clearly the optimal solution of Problem \eqref{alpha_primal} consists of $\Gamma$ variables at 1. Associating multipliers $\rho_i$ to the knapsack constraint and  $\sigma_{ij}$ for  $j \in J$ to the upper bound constraints, we consider the dual of Problem \eqref{alpha_primal}

\begin{equation}\label{alpha_dual}
\begin{array}{rlr}
     \min & \Gamma \rho_i + \displaystyle \sum_{j \in J} \sigma_{ij} \\
    s.t. & \rho_i + \sigma_{ij} \geq \displaystyle \hat d_j x_{ij}^{\star} 
    & \quad j \in J\\
    & \sigma_{ij} \geq 0 & \quad j \in J\\
    & \rho_i \geq 0. 
    \end{array}
\end{equation}
By strong duality, since Problem \eqref{alpha_primal} is feasible and bounded for all integer $\Gamma \in [0,|J|]$, then the dual problem \eqref{alpha_dual} is also feasible and bounded and their objective functions assume the same value in an optimal solution. Hence, $\alpha(x_i^{\star})$ is equal to the optimal value of Problem  \eqref{alpha_dual} for each $i \in I$.
Thus constraints \eqref{constraint_on_v} can be equivalently formulated as
$$
\begin{array}{cc}
\displaystyle  v_i - \sum_{j \in J} d_j x_{ij} - \left(\Gamma \rho_i + \sum_{j \in J} \sigma_{ij}\right) \geq 0 &i \in I   \\
      \rho_i + \sigma_{ij} \geq \displaystyle \hat d_j x_{ij}   & i \in I, \  j \in J\\
     \sigma_{ij} \geq 0 &  i \in I, \ j \in J\\
     \rho_i \geq 0 &i \in I.  
\end{array}
$$

\bigskip

Then, given vectors $x_i^{\star}$ for all $i\in I$, we define the total protection function $\beta$ representing the sum of the $\Gamma$ largest deviations  as:
\begin{equation*}
    \beta(x^{\star}) = \max_{\{S: \ S \subseteq  J, |S| \leq \Gamma\}}\left\{\sum_{i \in I} \sum_{j \in S} \hat d_j x^{\star}_{ij} \right\}.
\end{equation*}
By introducing the binary variables $w_{ij}$, the protection function $\beta$ can be equivalently formulated as
\begin{equation*}
\begin{array}{rlr}
    \beta(x^{\star}) = \max & \displaystyle \sum_{i \in I} \sum_{j \in J} \hat d_j x^{\star}_{ij}w_j\\
    s.t. & \displaystyle \sum_{j \in  J} w_j \leq \Gamma\\
    & w_j \in \{0, 1\} & \quad j \in J
    \end{array}
\end{equation*}
which equals to
\begin{equation}\label{beta_primal_2}
\begin{array}{rlr}
    \beta(x^{\star}) = \max & \displaystyle \sum_{i \in I} \sum_{j \in J} \hat d_j x^{\star}_{ij}w_j\\
    s.t. & \displaystyle \sum_{j \in J} w_j \leq \Gamma \\ 
    & 0 \leq w_j \leq 1 & \quad j \in J 
    \end{array}
\end{equation}
where we relaxed the integrality of variables $w_{ij}$ since the matrix of the constraints is totally unimodular and $\Gamma$ is integral (linear programs of this form have integral optima). 
Clearly the optimal solution of Problem \eqref{beta_primal_2} consists of $\Gamma$ variables at 1. 
We next consider the dual of Problem \eqref{beta_primal_2}. Using the multipliers $\tau\in \R$ for the knapsack constraint and $\pi\in\R^{|J|}$ for the upper bound constraints, we get
\begin{equation}\label{beta_dual_2}
\begin{array}{rlr}
    \min & \Gamma \tau + \displaystyle \sum_{j \in J} \pi_j \\
    s.t. & \tau + \pi_j \geq \displaystyle \sum_{i \in I} \hat d_j x_{ij}^{\star} 
    & \quad j \in J\\
    & \pi_j \geq 0 & \quad j \in J\\
    & \tau \geq 0. 
    \end{array}
\end{equation}
By strong duality, since Problem \eqref{beta_primal_2} is feasible and bounded for all $\Gamma \in [0,|J|]$, then the dual problem \eqref{beta_dual_2} is also feasible and bounded and their objective functions assume the same value in an optimal solution. Hence, $\beta(x^{\star})$ is equal to the optimal value of Problem \eqref{beta_dual_2}.
Thus constraint \eqref{constraint_on_D} can be equivalently formulated as
$$
\begin{array}{lr}
\displaystyle \sum_{i \in I} \sum_{j \in J} d_j x_{ij} - \left(\Gamma \tau + \sum_{j \in J} \pi_j\right) \geq D    &  \\
\tau + \pi_j \geq \displaystyle \sum_{i \in I} \hat d_j x_{ij} 
    & \quad j \in J\\
     \pi_j \geq 0 & \quad j \in J\\
     \tau \geq 0. 
\end{array}
$$
We finally get formulation \eqref{whole_extended}.
\end{proof}
\begin{flushright}
$\square$
\end{flushright}

\section{Appendix}
\label{sec:appendix_tables}

Tables \ref{tab:instances} report the characteristics of the instances of congested partial set covering location. Specifically, for each instance, the following data are reported: a unique identifier ID, a number s linked to the seed of the instance generator, the maximum number of customer deviations $\Gamma$ we consider for the robust setting, the maximum distance of coverage R, the optimal value for instance at hand ObjVal, and the number of open facilities \#Fac at the best feasible solution found on the instance.

\begin{table}[h]
\caption{Characteristics of the testbed.}\label{tab:instances}
\resizebox{\textwidth}{!}{
\begin{tabular}{lrrrrrrlrrrrrllrrrrr}\toprule
\textbf{ID} & \textbf{s} & \textbf{$\bm\Gamma$} & \textbf{R} & \textbf{ObjVal} & \textbf{\#Fac} &  & \multicolumn{1}{c}{\textbf{ID}} & \textbf{s} & \textbf{$\bm\Gamma$} & \textbf{R} & \textbf{ObjVal} & \textbf{\#Fac} &  & \multicolumn{1}{c}{\textbf{ID}} & \textbf{s} & \textbf{$\bm\Gamma$} & \textbf{R} & \textbf{ObjVal} & \textbf{\#Fac} \\
\cmidrule(lr){1-6}\cmidrule(lr){7-13}\cmidrule(lr){14-20}
1           & 1          & 25             & 5.5        & 27239.74        & 41             &  & 43                              & 3          & 100            & 6          & 29245.32        & 36             &  & 86                              & 4          & 300            & 5.75       & 30739.45        & 38             \\
2           & 1          & 25             & 5.75       & 27162.26        & 41             &  & 44                              & 3          & 100            & 6.25       & 29197.75        & 36             &  & 87                              & 4          & 300            & 6          & 30723.60        & 38             \\
3           & 1          & 25             & 6          & 27080.35        & 41             &  & 45                              & 3          & 150            & 5.5        & 29563.69        & 37             &  & 88                              & 4          & 300            & 6.25       & 30707.81        & 38             \\
4           & 1          & 25             & 6.25       & 27005.71        & 41             &  & 46                              & 3          & 150            & 5.75       & 29509.88        & 36             &  & 89                              & 4          & 400            & 5.5        & 30968.95        & 38             \\
5           & 1          & 50             & 5.5        & 27952.05        & 41             &  & 47                              & 3          & 150            & 6          & 29497.71        & 36             &  & 90                              & 4          & 400            & 5.75       & 30941.75        & 38             \\
6           & 1          & 50             & 5.75       & 27856.05        & 41             &  & 48                              & 3          & 150            & 6.25       & 29466.02        & 37             &  & 91                              & 4          & 400            & 6          & 30923.64        & 38             \\
7           & 1          & 50             & 6          & 27769.04        & 41             &  & 49                              & 3          & 200            & 5.5        & 29766.99        & 37             &  & 92                              & 4          & 400            & 6.25       & 30905.46        & 38             \\
8           & 1          & 50             & 6.25       & 27669.14        & 41             &  & 50                              & 3          & 200            & 5.75       & 29705.32        & 37             &  & 93                              & 4          & 500            & 5.5        & 31119.74        & 39             \\
9           & 1          & 100            & 5.5        & 28450.44        & 42             &  & 51                              & 3          & 200            & 6          & 29692.82        & 37             &  & 94                              & 4          & 500            & 5.75       & 31094.42        & 38             \\
10          & 1          & 100            & 5.75       & 28406.33        & 42             &  & 52                              & 3          & 200            & 6.25       & 29660.42        & 37             &  & 95                              & 4          & 500            & 6          & 31076.66        & 38             \\
11          & 1          & 100            & 6          & 28368.57        & 42             &  & 53                              & 3          & 300            & 5.5        & 30089.87        & 37             &  & 96                              & 4          & 500            & 6.25       & 31057.25        & 38             \\
12          & 1          & 100            & 6.25       & 28316.45        & 42             &  & 54                              & 3          & 300            & 5.75       & 30022.47        & 37             &  & 97                              & 5          & 25             & 5.5        & 27614.20        & 40             \\
13          & 1          & 150            & 5.5        & 28680.39        & 42             &  & 55                              & 3          & 300            & 6          & 30010.20        & 37             &  & 98                              & 5          & 25             & 5.75       & 27515.77        & 40             \\
14          & 1          & 150            & 5.75       & 28631.80        & 42             &  & 56                              & 3          & 300            & 6.25       & 29979.06        & 37             &  & 99                              & 5          & 25             & 6          & 27421.87        & 40             \\
15          & 1          & 150            & 6          & 28597.00        & 42             &  & 57                              & 3          & 400            & 5.5        & 30285.64        & 37             &  & 100                             & 5          & 25             & 6.25       & 27346.55        & 40             \\
16          & 1          & 150            & 6.25       & 28557.92        & 42             &  & 58                              & 3          & 400            & 5.75       & 30214.90        & 37             &  & 101                             & 5          & 50             & 5.5        & 28300.76        & 41             \\
17          & 1          & 200            & 5.5        & 28898.38        & 42             &  & 59                              & 3          & 400            & 6          & 30202.67        & 37             &  & 102                             & 5          & 50             & 5.75       & 28207.67        & 40             \\
18          & 1          & 200            & 5.75       & 28845.45        & 42             &  & 60                              & 3          & 400            & 6.25       & 30169.95        & 37             &  & 103                             & 5          & 50             & 6          & 28112.83        & 40             \\
19          & 1          & 200            & 6          & 28807.26        & 42             &  & 61                              & 3          & 500            & 5.5        & 30450.82        & 37             &  & 104                             & 5          & 50             & 6.25       & 28034.17        & 40             \\
20          & 1          & 200            & 6.25       & 28763.63        & 43             &  & 62                              & 3          & 500            & 5.75       & 30372.28        & 37             &  & 105                             & 5          & 100            & 5.5        & 28827.47        & 41             \\
21          & 1          & 300            & 5.5        & 29228.79        & 42             &  & 63                              & 3          & 500            & 6          & 30356.26        & 37             &  & 106                             & 5          & 100            & 5.75       & 28801.14        & 41             \\
22          & 1          & 300            & 5.75       & 29168.65        & 42             &  & 64                              & 3          & 500            & 6.25       & 30321.74        & 37             &  & 107                             & 5          & 100            & 6          & 28763.12        & 41             \\
23          & 1          & 300            & 6          & 29123.02        & 42             &  & 65                              & 4          & 25             & 5.5        & 28769.80        & 36             &  & 108                             & 5          & 100            & 6.25       & 28716.61        & 41             \\
24          & 1          & 300            & 6.25       & 29072.71        & 43             &  & 66                              & 4          & 25             & 5.75       & 28688.22        & 36             &  & 109                             & 5          & 150            & 5.5        & 29086.06        & 41             \\
25          & 1          & 400            & 5.5        & 29477.27        & 43             &  & 67                              & 4          & 25             & 6          & 28626.61        & 36             &  & 110                             & 5          & 150            & 5.75       & 29060.88        & 41             \\
26          & 1          & 400            & 5.75       & 29401.55        & 42             &  & 68                              & 4          & 25             & 6.25       & 28566.13        & 36             &  & 111                             & 5          & 150            & 6          & 29031.55        & 41             \\
27          & 1          & 400            & 6          & 29350.02        & 43             &  & 69                              & 4          & 50             & 5.5        & 29455.33        & 37             &  & 112                             & 5          & 150            & 6.25       & 29000.45        & 41             \\
28          & 1          & 400            & 6.25       & 29295.53        & 43             &  & 70                              & 4          & 50             & 5.75       & 29349.76        & 37             &  & 113                             & 5          & 200            & 5.5        & 29310.83        & 41             \\
29          & 1          & 500            & 5.5        & 29634.08        & 43             &  & 71                              & 4          & 50             & 6          & 29265.07        & 37             &  & 114                             & 5          & 200            & 5.75       & 29284.08        & 41             \\
30          & 1          & 500            & 5.75       & 29559.19        & 43             &  & 72                              & 4          & 50             & 6.25       & 29186.63        & 37             &  & 115                             & 5          & 200            & 6          & 29253.31        & 41             \\
31          & 1          & 500            & 6          & 29504.86        & 43             &  & 73                              & 4          & 100            & 5.5        & 30005.30        & 38             &  & 116                             & 5          & 200            & 6.25       & 29227.89        & 41             \\
32          & 1          & 500            & 6.25       & 29450.37        & 43             &  & 74                              & 4          & 100            & 5.75       & 29978.57        & 37             &  & 117                             & 5          & 300            & 5.5        & 29637.01        & 42             \\
33          & 3          & 25             & 5.5        & 28108.43        & 35             &  & 75                              & 4          & 100            & 6          & 29954.70        & 37             &  & 118                             & 5          & 300            & 5.75       & 29608.42        & 41             \\
34          & 3          & 25             & 5.75       & 28018.51        & 35             &  & 76                              & 4          & 100            & 6.25       & 29926.45        & 37             &  & 119                             & 5          & 300            & 6          & 29577.26        & 41             \\
35          & 3          & 25             & 6          & 27951.82        & 35             &  & 77                              & 4          & 150            & 5.5        & 30240.51        & 38             &  & 120                             & 5          & 300            & 6.25       & 29544.77        & 42             \\
36          & 3          & 25             & 6.25       & 27886.82        & 35             &  & 78                              & 4          & 150            & 5.75       & 30219.40        & 37             &  & 121                             & 5          & 400            & 5.5        & 29863.25        & 42             \\
37          & 3          & 50             & 5.5        & 28789.70        & 36             &  & 79                              & 4          & 150            & 6          & 30204.67        & 37             &  & 122                             & 5          & 400            & 5.75       & 29832.83        & 42             \\
38          & 3          & 50             & 5.75       & 28683.83        & 36             &  & 80                              & 4          & 150            & 6.25       & 30190.79        & 37             &  & 123                             & 5          & 400            & 6          & 29797.02        & 42             \\
39          & 3          & 50             & 6          & 28608.38        & 36             &  & 81                              & 4          & 200            & 5.5        & 30441.22        & 38             &  & 124                             & 5          & 400            & 6.25       & 29763.65        & 42             \\
40          & 3          & 50             & 6.25       & 28531.91        & 36             &  & 82                              & 4          & 200            & 5.75       & 30418.71        & 38             &  & 125                             & 5          & 500            & 5.5        & 30022.97        & 42             \\
41          & 3          & 100            & 5.5        & 29319.41        & 36             &  & 83                              & 4          & 200            & 6          & 30403.79        & 38             &  & 126                             & 5          & 500            & 5.75       & 29992.16        & 42             \\
42          & 3          & 100            & 5.75       & 29264.81        & 36             &  & 84                              & 4          & 200            & 6.25       & 30388.13        & 37             &  & 127                             & 5          & 500            & 6          & 29957.95        & 42             \\
            &            &                &            &                 &                &  & 85                              & 4          & 300            & 5.5        & 30763.43        & 38             &  & 128                             & 5          & 500            & 6.25       & 29923.87        & 42            \\
\cmidrule(lr){1-6}\cmidrule(lr){7-13}\cmidrule(lr){14-20}
\end{tabular}}
\end{table}

Tables \ref{tab:results_CPSCLP_1}-\ref{tab:results_CPSCLP_3} report the results of four exact algorithms: Gurobi on the extended formulation (MIQP), Gurobi on the perspective reformulation (MISOCP), and two versions of our Benders algorithms directly on the perspective reformulation: a single-tree version using the $\epsilon$-technique (ST$\epsilon$-BEN), a multi-tree version using the $\epsilon$-technique (MT$\epsilon$-BEN). The metrics we use for the comparison are: i) the computational time (Time) expressed in seconds, ii) the relative gap at the end of the optimization (Gap), iii) the number (or the average number for the multi-tree approach) of branching nodes explored (Nodes) and iv) the number of Benders cuts generated on integral solutions (BCuts) and fractional solutions (FrBCuts). As for the Gap, the internal relative gap of Gurobi was used where available. For the single-tree method in the absence of feasible solutions and multi-tree method running out of time, the gap was computed as $ 100\frac{{\text{ObjVal} - \text{BestLB}}}{\text{ObjVal}}$, where ObjVal is the optimal value of the instance (i.e., it is the value reported under column ObjVal in Table \ref{tab:instances}), and BestLB is the best lower bound found by the method at hand for the same instance.

\begin{table}[h]
\caption{Performance of the exact procedures on the testbed (part I). Runs reaching the time limit of 900 seconds are indicated by “TL”.}\label{tab:results_CPSCLP_1}
\resizebox{\textwidth}{!}{
\begin{tabular}{lrrrrrrrrrrrrrrr} \toprule
            & \multicolumn{3}{c}{\textbf{MIQP}}                                                                                   & \multicolumn{3}{c}{\textbf{MISOCP}}                                                                                       & \multicolumn{5}{c}{\textbf{ST$\bm\epsilon$-BEN}}                                                                                                                                                                & \multicolumn{4}{c}{\textbf{MT$\bm\epsilon$-BEN}}                                                                                                                         \\\cmidrule(lr){2-4}\cmidrule(lr){5-7}\cmidrule(lr){8-12}\cmidrule(lr){13-16}
\textbf{ID} & \multicolumn{1}{c}{\textbf{Times{[}s{]}}} & \multicolumn{1}{c}{\textbf{Gap{[}\%{]}}} & \multicolumn{1}{c}{\textbf{Nodes}} & \multicolumn{1}{c}{\textbf{Times{[}s{]}}} & \multicolumn{1}{c}{\textbf{Gap{[}\%{]}}} & \multicolumn{1}{c}{\textbf{Nodes}} & \multicolumn{1}{c}{\textbf{Times{[}s{]}}} & \multicolumn{1}{c}{\textbf{Gap{[}\%{]}}} & \multicolumn{1}{c}{\textbf{Nodes}} & \multicolumn{1}{c}{\textbf{BCuts}} & \multicolumn{1}{c}{\textbf{FrBCuts}} & \multicolumn{1}{c}{\textbf{Times{[}s{]}}} & \multicolumn{1}{c}{\textbf{Gap{[}\%{]}}} & \multicolumn{1}{c}{\textbf{Nodes}} & \multicolumn{1}{c}{\textbf{BCuts}} \\\cmidrule(lr){1-1}\cmidrule(lr){2-4}\cmidrule(lr){5-7}\cmidrule(lr){8-12}\cmidrule(lr){13-16}
1                        & TL                                & 59.20                            & 1673                       & 258.71                            & 0                                & 102                        & 84.25                             & 0                                & 6769                       & 23                         & 3                            & 22.16                             & 0                                & 3833.67                    & 8                          \\
\bottomrule                               
\end{tabular}}
\end{table}

\begin{table}[h]
\caption{Performance of the exact procedures on the testbed (part II). Runs reaching the time limit of 900 seconds are indicated by “TL”.}\label{tab:results_CPSCLP_2}
\resizebox{\textwidth}{!}{
\begin{tabular}{lrrrrrrrrrrrrrrr} \toprule
            & \multicolumn{3}{c}{\textbf{MIQP}}                                                                                   & \multicolumn{3}{c}{\textbf{MISOCP}}                                                                                       & \multicolumn{5}{c}{\textbf{ST$\bm\epsilon$-BEN}}                                                                                                                                                                & \multicolumn{4}{c}{\textbf{MT$\bm\epsilon$-BEN}}                                                                                                                         \\\cmidrule(lr){2-4}\cmidrule(lr){5-7}\cmidrule(lr){8-12}\cmidrule(lr){13-16}
\textbf{ID} & \multicolumn{1}{c}{\textbf{Times{[}s{]}}} & \multicolumn{1}{c}{\textbf{Gap{[}\%{]}}} & \multicolumn{1}{c}{\textbf{Nodes}} & \multicolumn{1}{c}{\textbf{Times{[}s{]}}} & \multicolumn{1}{c}{\textbf{Gap{[}\%{]}}} & \multicolumn{1}{c}{\textbf{Nodes}} & \multicolumn{1}{c}{\textbf{Times{[}s{]}}} & \multicolumn{1}{c}{\textbf{Gap{[}\%{]}}} & \multicolumn{1}{c}{\textbf{Nodes}} & \multicolumn{1}{c}{\textbf{BCuts}} & \multicolumn{1}{c}{\textbf{FrBCuts}} & \multicolumn{1}{c}{\textbf{Times{[}s{]}}} & \multicolumn{1}{c}{\textbf{Gap{[}\%{]}}} & \multicolumn{1}{c}{\textbf{Nodes}} & \multicolumn{1}{c}{\textbf{BCuts}} \\\cmidrule(lr){1-1}\cmidrule(lr){2-4}\cmidrule(lr){5-7}\cmidrule(lr){8-12}\cmidrule(lr){13-16}
43 & TL           & 60.52       & 1934  & 405.90       & 0           & 46    & 232.86       & 0           & 28824 & 41    & 3       & 35.81        & 0           & 9505.25  & 7     \\
44 & TL           & 61.43       & 2072  & 391.95       & 0           & 83    & 169.47       & 0           & 19606 & 23    & 3       & 46.80        & 0           & 15093.90 & 9     \\
\bottomrule                              
\end{tabular}}
\end{table}

\begin{table}[h]
\caption{Performance of the exact procedures on the testbed (part III). Runs reaching the time limit of 900 seconds are indicated by “TL”.}\label{tab:results_CPSCLP_3}
\resizebox{\textwidth}{!}{
\begin{tabular}{lrrrrrrrrrrrrrrr} \toprule
            & \multicolumn{3}{c}{\textbf{MIQP}}                                                                                   & \multicolumn{3}{c}{\textbf{MISOCP}}                                                                                       & \multicolumn{5}{c}{\textbf{ST$\bm\epsilon$-BEN}}                                                                                                                                                                & \multicolumn{4}{c}{\textbf{MT$\bm\epsilon$-BEN}}                                                                                                                         \\\cmidrule(lr){2-4}\cmidrule(lr){5-7}\cmidrule(lr){8-12}\cmidrule(lr){13-16}
\textbf{ID} & \multicolumn{1}{c}{\textbf{Times{[}s{]}}} & \multicolumn{1}{c}{\textbf{Gap{[}\%{]}}} & \multicolumn{1}{c}{\textbf{Nodes}} & \multicolumn{1}{c}{\textbf{Times{[}s{]}}} & \multicolumn{1}{c}{\textbf{Gap{[}\%{]}}} & \multicolumn{1}{c}{\textbf{Nodes}} & \multicolumn{1}{c}{\textbf{Times{[}s{]}}} & \multicolumn{1}{c}{\textbf{Gap{[}\%{]}}} & \multicolumn{1}{c}{\textbf{Nodes}} & \multicolumn{1}{c}{\textbf{BCuts}} & \multicolumn{1}{c}{\textbf{FrBCuts}} & \multicolumn{1}{c}{\textbf{Times{[}s{]}}} & \multicolumn{1}{c}{\textbf{Gap{[}\%{]}}} & \multicolumn{1}{c}{\textbf{Nodes}} & \multicolumn{1}{c}{\textbf{BCuts}} \\\cmidrule(lr){1-1}\cmidrule(lr){2-4}\cmidrule(lr){5-7}\cmidrule(lr){8-12}\cmidrule(lr){13-16}
86  & TL           & 59.46       & 3217  & 345.93       & 0           & 127   & 46.74        & 0           & 8314  & 31    & 3       & 65.00        & 0           & 12021.62 & 12    \\
87  & TL           & 60.28       & 1855  & 358.78       & 0           & 204   & 20.30        & 0           & 5686  & 12    & 3       & 59.06        & 0           & 9716.25  & 11    \\
\bottomrule                              
\end{tabular}}
\end{table}

\end{document}